\documentclass[12pt]{article}

\RequirePackage[OT1]{fontenc}
\RequirePackage[numbers]{natbib}

\usepackage{graphicx,amsmath,amsfonts,amssymb,amsthm,psfrag, color}
\usepackage{comment}
\usepackage{calc}
\usepackage{wrapfig,float}
\usepackage{fullpage}

\graphicspath{{file/}}
\newcommand{\ds}{\displaystyle}

\newcommand{\dint}{\mathrm{d}}

\includecomment{comment} 

\newtheorem{thm}{Theorem}[section]

\newtheorem{lemma}[thm]{Lemma}
\newtheorem{proposition}[thm]{Proposition}

\newtheorem{definition}[thm]{Definition}

\newtheorem{rem}[thm]{Remark}
\numberwithin{equation}{section}

\title{Initial-Boundary Value Problem for the heat equation -\\ A stochastic algorithm}

\begin{document}
\author{Madalina Deaconu$^1$ and Samuel Herrmann$^2$\\[5pt]
\small {$^1$Inria, Villers-l{\`e}s-Nancy, F-54600, France;}\\
\small{Universit{\'e} de Lorraine, CNRS,
Institut Elie Cartan de Lorraine - UMR 7502,}\\ 
\small{Vandoeuvre-l{\`e}s-Nancy, F-54506, France}\\
\small{Madalina.Deaconu@inria.fr}\\[5pt]
\small{$^2$Institut de Math{\'e}matiques de Bourgogne (IMB) - UMR 5584, CNRS,}\\
\small{Universit{\'e} de Bourgogne Franche-Comt\'e, F-21000 Dijon, France} \\
\small{Samuel.Herrmann@u-bourgogne.fr}
}

\maketitle

\begin{abstract}
The Initial-Boundary Value Problem for the heat equation is solved by using a new algorithm based on a random walk on heat balls. Even if it represents a sophisticated generalization of the Walk on Spheres (WOS) algorithm introduced to solve the Dirichlet problem for Laplace's equation, its implementation is rather easy.  The definition of the random walk is based on a new mean value formula for the heat equation. The convergence results and numerical examples  permit to emphasize the efficiency and accuracy of the algorithm. 
\end{abstract}

\par\medskip


\section{Introduction}
In this paper, we  study the Initial-Boundary Value Problem (IBVP) associated to the heat equation and develop a new method of simulation based on the Walk on Moving Sphere Algorithm (WOMS).
The main objective is to construct an efficient approximation to the solution of the IBVP. The solution is a $\mathcal{C}^{1,2}$ function $u$ satisfying
\begin{equation}
\label{eq:DP}
\left\{\begin{array}{ll}
\partial_t u(t,x)=\Delta_x u(t,x), & \forall (t,x)\in \mathbb{R}_+\times\mathcal{D,}\\
u(t,x)=f(t,x),&\forall (t,x)\in\mathbb{R}_+ \times\partial\mathcal{D},\\
u(0,x)=f_0(x), & \forall\ x\in\mathcal{D},
\end{array}\right.
\end{equation}
where $f$ is a continuous function defined on $\mathbb{R}_+ \times\partial\mathcal{D}$, $f_0$ is continuous on $\mathcal{D}$  and $\mathcal{D}$ denotes a bounded finitely connected domain in $\mathbb{R}^d$. For compatibility reasons we have also $f(0,x)=f_0(x)$.

The foundation stone of our work is the probabilistic representation for the solution of a partial differential equation.  
Suppose that we are looking for the solution $u(t,x)$ of some PDE defined on the whole space $\mathbb{R}^d$. Under suitable hypothesis we can use the classical form
$u(t,x)=\mathbb{E} [f(t,X_t)]$ where $(X_t)_{t\in\mathbb{R}_+}$ is a stochastic process, satisfying a stochastic differential equation, and $f$ a known function.
In order to approximate $u(t,x)$, the Strong Law of large Number allows us to construct Monte Carlo methods once we are able to propose an
approximating procedure for the stochastic process $(X_t)_{t\in\mathbb{R_+}}$.

The problem is more difficult  when considering problems with boundary conditions. Nevertheless if some regularity is provided we can also find a probabilistic approach. A generic representation, for the solution of the Dirichlet problem in a domain $\mathcal{D}$ (the solution does not depend on time), is
$$u(x) = \mathbb{E}\left[ f(X_{\tau_{\mathcal{D}}})\exp\left(-\int_0^{\tau_{\mathcal{D}}} k(X_s) \dint s\right )-\int_0^{\tau_\mathcal{D}} g(X_t)\exp\left(-\int_0^tk(X_s)\dint s \right )\dint t\right],$$
where $f,\, g,\, k$ are given functions, $X_0=x$  and $\tau_\mathcal{D}=\inf\{ t\geq 0;\, X_t\in \partial \mathcal{D}\}$.  We refer to several classical books for more details \cite{Bass, K-S, Friedman2, Oksendal}. The problem is hard to address as, in order to give an approximation, we need to approach the hitting time, the exit position and sometimes even the path of the process $X_t$ up to exit the domain
 $\mathcal{D}$. 
 
In particular situations we need to characterize either the hitting time $
\tau_\mathcal{D}$ or the exit position $X_{\tau_\mathcal{D}}$, and these problems reveal quite difficult. The main goal of our work is to handle  a more complex situation by unearthing numerical algorithms for the couple $(\tau_\mathcal{D},X_{\tau_\mathcal{D}})$ itself.

To fix ideas and present a brief history, consider the simple Dirichlet problem for Laplace's equation in a smooth and bounded domain $\mathcal{D}\subset\mathbb{R}^d$:
\begin{equation*}
\left\{
\begin{array}{ll}
\Delta u (x)  = 0, &\forall x\in \mathcal{D}\\
u(x)=f(x), & \forall x\in  \partial \mathcal{D}.\\
\end{array}
\right.
\end{equation*}
We recall the associated probabilistic representation: $u(x)=\mathbb{E}_x[f(X_{\tau_{\mathcal{D}}})]$  where $(X_t,\, t\ge 0)$ here stands for the $d$-dimensional Brownian motion starting in $x$.
The original idea in order to approximate $u(x)$ by using the walk on spheres algorithm (WOS) goes back to M\"{u}ller \cite{muller_56}. The idea consists in constructing a step by step $\mathbb{R}^d$-valued Markov chain $(x_n,\ n\ge 0)$ with initial point $x_0:=x$  which converges towards a limit $x_\infty$, $x_\infty$ and $X_{\tau_\mathcal{D}}$ being identically distributed. Let us roughly describe $(x_n)$: first, we choose $S_0$ the largest sphere centered in $x_0$ and included in $\mathcal{D}$. The first exit point $x_1$
from the sphere $S_0$ for the Brownian motion starting from $x_0$ has an uniform distribution on $\partial S_0$ and is easy to sample. 
\begin{figure}
\begin{psfrags}
\psfrag{0}{$x_0$}
\psfrag{1}{$x_1$}
\psfrag{2}{$x_2$}
\psfrag{3}{$x_3$}
\psfrag{4}{$x_4$}
\psfrag{d}{$\mathcal{D}$}
\centerline{\includegraphics[scale=0.5]{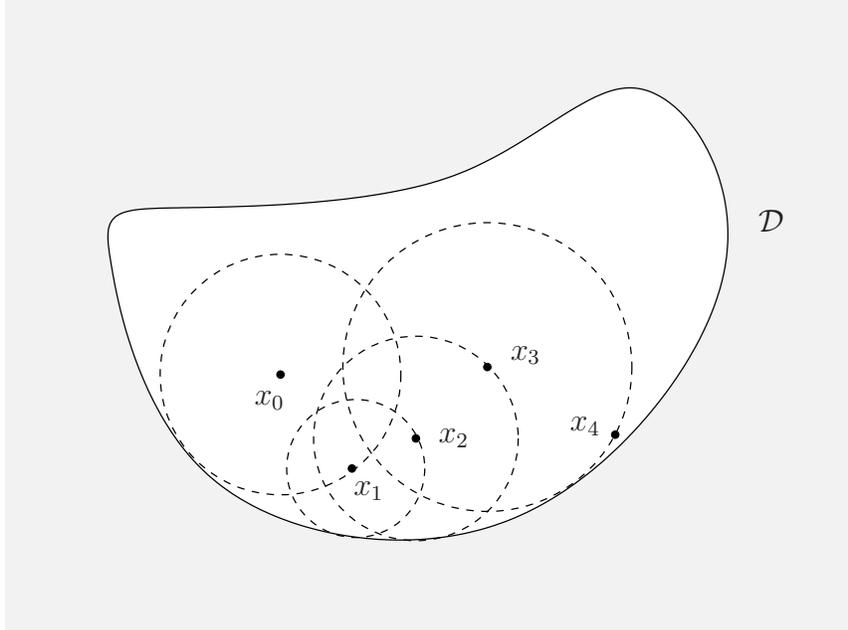}}
\end{psfrags}
\caption{WOS algorithm}
\label{WOS-figure}
\end{figure}

The construction is pursued with the new starting point given by $x_1$ (see Figure \ref{WOS-figure}). The algorithm goes on and stops while reaching the boundary $\partial \mathcal{D}$. In order to avoid an infinite sequence of hitting times the stopping criteria of the algorithm includes a $\varepsilon$ test: we stop the Markov chain as soon as $ \delta (x_n,\partial \mathcal{D}) \leq \varepsilon $ ($\delta $ represents here the Euclidean distance in $\mathbb{R}^d$).   Convergence results depending on $\varepsilon$ and on the regularity of $\partial \mathcal{D}$ can be found in M\"{u}ller \cite{muller_56} and  Mascagni and Hwuang \cite{mascagni-hwang-2003}. 
Generalization of this result to a constant drift, by means of convergence theorems for discrete-time martingales, was proposed in the work of Villa-Moral\`es \cite{villa-2012}, \cite{villa-2016}.  Binder and Braverman \cite{Binder-Bravermann} gave also the complete characterization of the rate of convergence for the WOS in terms of the local geometry  of $\mathcal{D}$. Other elliptic problems have been studied by Gurov, Whitlock and Dimov \cite{Gurov2001}.

If needed, we can also approach the boundary hitting time by using the explicit form of  its probability distribution function.  However, a real difficult leap appears when we want to move from the simulation of $X_t$ to the simulation of $(t,X_t)$. For example, if the domain is a sphere then $X_{\tau_\mathcal{D}}$ can be simulated by the uniform random variable on the $\partial \mathcal{D}$ while $\tau_\mathcal{D}$ has an explicit pdf function which is not suited for numerical approaches as it depends on the Bessel function.

In previous works \cite{deaconu-herrmann-2013}, \cite{deaconu-herrmann-2016}, \cite{deaconu-herrmann-2015}  the authors discussed the connexion between the hitting times of the Bessel process and Brownian ones and introduced a new technique for approximating both the hitting time and the exit position.
These previous studies on the hitting time form the foundation of our current work. We propose a new algorithm, involving a random walk on heat balls belonging to the domain $[0,t]\times\mathcal{D}$ (see \citep{evans-2010} p.53 for a definition of the heat ball) which approaches $(\tau_D, X_{\tau_D})$ in general domains. Thus we obtain a method for approximating the solution of the equation \eqref{eq:DP}.  Let us mention at this stage that Sabelfeld \cite{Sabelfeld} already described a random walk $(\tau_n,Y_n)_{n\ge 0}$ for solving the Initial-Boundary Value Problem for the heat equation. His approach is essentially different: first of all his random walk is valued only on the boundary $[0,t]\times\partial \mathcal{D}$ and secondly the main argument is based on solving an integral equation of the second kind rather than using Monte-Carlo techniques. A nicely written description of the method can be find in \cite{sabelfeld-simonov-1994}. Let us just note that such algorithm which permits to evaluate $u(t,x)$ is less accurate for large time $t$ or non convex domains $\mathcal{D}$. 

Let us now introduce the main results concerning the algorithm Random Walk on Heat Balls which approximates $(\tau_{\mathcal{D}}, X_{\tau_{\mathcal{D}}})$, $X$ being a $d$-dimensional Brownian motion. We just introduce first some preliminary notations: we recall that $\delta(x,\partial\mathcal{D})$ is the Euclidean distance between the point $x$ and the boundary of the domain and introduce the function $\alpha(u,v)=\min(u,\frac{e}{2d}\,\delta^2(v,\partial \mathcal{D}) )$. In the following, $(U_n)_{n\ge 1}$ stands for a sequence of independent uniformly distributed random vectors on $[0,1]^{\lfloor d/2\rfloor+1}$, $\Pi^U_n$ denote the product of all its coordinates, $(G_n)_{n\ge 1}$ is a sequence of independent standard Gaussian r.v.  and $(V_n)_{n\ge 1}$ a sequence of independent uniformly distributed random vectors on the unit sphere of dimension $d$, centered on the origin. We assume these three sequences to be independent. Let us define:
\[
R_{n+1}:=\Big(\Pi^U_{n+1}\Big)^{2/d}\exp\Big\{-(1-\frac{2}{d}\lfloor \frac{d}{2}\rfloor )G_{n+1}^2\Big\}
\]
and construct a sequence $(T_n,X_n)_{n\ge 0}$ by the following procedure (Figure \ref{WOMS-figure}).\\[5pt]
\begin{figure}
\begin{psfrags}
\psfrag{(t,x)}{{\scriptsize $(T_0,X_0)$}}
\psfrag{time}{{\scriptsize Time axis}}
\psfrag{A}{{\scriptsize $(T_1,X_1)$}}
\psfrag{B}{{\scriptsize $(T_2,X_2)$}}
\psfrag{d}{$\mathcal{D}$}
\centerline{\includegraphics[scale=0.5]{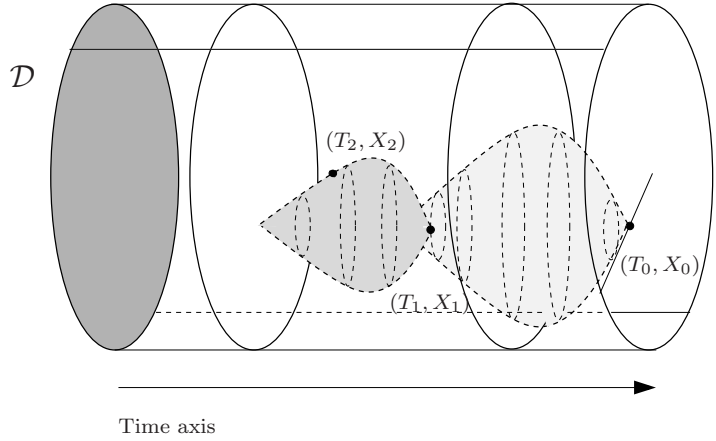}}
\end{psfrags}
\caption{Markov chain $(T_n,X_n)_{n\ge 0}$}
\label{WOMS-figure}
\end{figure}

\vspace*{0.5cm}
\centerline{{\sc\bf ALGORITHM}}
\vspace*{0.2cm}
\noindent\fbox{\parbox{\linewidth-2\fboxrule-2\fboxsep}{\vspace*{0.1cm}
{\bf Initialisation:}  Fix $\varepsilon >0$.
The initial value of the sequence $(T_n,X_n)$ is  $(T_0,X_0)=(t,x)$. \\
{\bf Step n:}
The sequence is defined by recurrence as follows: for $n\ge 0$, 
\begin{align*}
\left\{\begin{array}{l}
T_{n+1}=T_n-\alpha(T_n,X_n) R_{n+1},\\[5pt]
X_{n+1}=X_n+2\sqrt{\alpha(T_n,X_n)}\psi_d(R_{n+1})V_{n+1},
\quad
\mbox{where}\quad \psi_d(t)=\sqrt{t\log(t^{-d/2})}.\end{array}\right.
\end{align*}}}

\vspace*{0.5cm}
\noindent\fbox{\parbox{\linewidth-2\fboxrule-2\fboxsep}{\vspace*{0.1cm}
{\bf Stop} If $\alpha(T_n,X_n) \leq \varepsilon  $ then ${\cal{N}}_\varepsilon = n$ 
\begin{enumerate}
\item If $\delta^2(X_{\mathcal{N}_\varepsilon},\partial \mathcal{D})\le \frac{2\varepsilon d}{e}$ then choose $X_\varepsilon\in \partial \mathcal{D}$ such that 
$
\delta(X_{\mathcal{N}_\varepsilon}, X_\varepsilon )=\delta(X_{\mathcal{N}_\varepsilon},\partial \mathcal{D})
$
and define $T_\varepsilon:=T_{\mathcal{N}_\varepsilon}$.
\item If $\delta^2(X_{\mathcal{N}_\varepsilon},\partial \mathcal{D})>\frac{2\varepsilon d}{e}$ then set $T_\varepsilon=0$ and $X_\varepsilon:=X_{\mathcal{N}_\varepsilon}$.
\end{enumerate}}}

\vspace{0.2cm}
\noindent {\bf Algorithm outcomes:} We get thus $(T_\varepsilon, X_\varepsilon)$ and $\mathcal{N}_\varepsilon$ the number of steps.\\[5pt]

We propose an approximation of the solution to \eqref{eq:DP} by using the definition:
\begin{equation*}
u^\varepsilon(t,x)=\mathbb{E}_{(t,x)}[f(T_\varepsilon,X_\varepsilon)1_{\{ X_\varepsilon\in\partial \mathcal{D} \}}]+\mathbb{E}_{(t,x)}[f_0(X_\varepsilon)1_{\{  X_\varepsilon\notin\partial \mathcal{D} \}}], \quad \mbox{for}\quad (t, x)\in[0,T]\times\overline{\mathcal{D}}.
\end{equation*}

\noindent We will prove the convergence of this approximation in Proposition \ref{thm:approx}:\\[5pt]
\noindent {\bf Convergence result.}
 \emph{Let us assume that the Initial-boundary Value Problem \eqref{eq:DP} admits an unique $\mathcal{C}^{1,2}([0,T]\times\mathcal{D})$-solution $u$, defined by \eqref{eq:prop:uniq}. We introduce the approximation $u^\varepsilon$ given by \eqref{eq:approx}. Then $u^\varepsilon$ converges towards $u$, as $\varepsilon\to 0$, uniformly with respect to $(t,x)$. Moreover there exist $\kappa_{T,\mathcal{D}}(u)>0$ and $\varepsilon>0$ such that
\[
|u(t,x)-u^\varepsilon(t,x)|\le \kappa_{T,\mathcal{D}}(u)\sqrt{\varepsilon},\quad \forall \varepsilon\le \varepsilon_0,\ (t,x)\in [0,T]\times\mathcal{D}.
\]}
\noindent An important result, based on the construction of a submartingale related to the Riesz potential, completes the convergence of the algorithm:\\[5pt]
\noindent {\bf Efficiency result.}
\emph{Let $\mathcal{D}\subset B(0,1)$ be a $0$-thick domain.
The number of steps $\mathcal{N}_\varepsilon$, of the approximation algorithm, is almost surely finite. Moreover there exist constants $C>0$  and $\varepsilon_0>0$ both independent of $(t,x)$ such that
\begin{equation*}
\label{eq:mean-number-0}
\mathbb{E}[\mathcal{N}_\varepsilon]\le C|\log \varepsilon|,\quad \mbox{for all }\ \varepsilon\le \varepsilon_0.
\end{equation*}}
The material is organized as follows. In the second section we present mean value properties for the heat equation which plays a central role in the definition of the algorithm. The third section constructs the Random Walk on Heat Balls used to solve the  Initial-Boundary Value Problem. In Section 4, we introduce the stopping procedure of the algorithm and prove the convergence result. The rate of the algorithm is also analyzed. We end up the paper with numerical results for two particular domains. These illustrations corroborate the accuracy of the algorithm.

\section{A mean value property associated to the heat equation}
In this section we will discuss the link between solutions of the heat equation and a particular version of the mean value property. This link is also an essential tool in the study of the classical Dirichlet problem. 

Let us first note that due to the time reversion, the solution of the Initial-Boundary Value Problem for the heat equation is directly related to the solution of the Terminal-Boundary value problem for the backward heat equation (heat equation with negative diffusion).
Due to this essential property, we are going to first present a mean value property for the \emph{backward heat equation} and then deduce a similar property for the heat equation.

Let $A$ be an open non empty set of $\mathbb{R}_+\times \mathbb{R}^d$.
\begin{definition} 
A function $h:A\mapsto \mathbb{R}$ is said to be a \emph{reverse temperature} in $A$ if $h$ is a $\mathcal{C}^{1,2}$-function satisfying
\begin{equation}
\label{def:space-time}
\partial_t h(t,x)+\Delta_x h(t,x)=0,\quad \forall (t,x)\in A.
\end{equation}
\end{definition}
\begin{proposition}\label{prop:me}
Let $A\subset \mathbb{R}_+\times \mathbb{R}^d$ be a non empty open set. If a function $h$ is a \emph{reverse temperature} in $A$ then it has the following mean value property:
\begin{equation}
\label{eq:mean-v}
h(t,x)=\frac{1}{2\pi^{d/2}}\iint_{(s,y)\in ]0,1[\times\mathbb{S}^d}\frac{1}{s}\, h(t+\alpha s,x+2\sqrt{\alpha}\psi_d(s)y)\,\psi_d^d(s)\,\dint \sigma(y)\,\dint s,
\end{equation}
where $\mathbb{S}^d$ is the $d$-dimensional sphere of radius $1$, $\sigma$ is the Lebesgue measure on $\mathbb{S}^d$ and 
\begin{equation}
\label{eq:psi}
\psi_d(t)=\sqrt{t\log (t^{-d/2})}, \quad t\in]0,1[.
\end{equation}
Equation \eqref{eq:mean-v} is satisfied for any $\alpha>0$ such that $[t,t+\alpha]\times \overline{B(x,2\sqrt{\alpha d/(2e)})}\subset A$. Here $\overline{B(x,r)}$ stands for the Euclidean ball centered in $x$ of radius $r$, i.e. $\overline{B(x,r)}=\{x\in\mathbb{R}^d\ s.t.\ \Vert x\Vert\le r\}$.
\end{proposition}
The mean value formula \eqref{eq:mean-v} is quite different and more general than the classical formula associated to the heat equation (see, for instance, Theorem 3 on page 53 in \cite{evans-2010}). Nevertheless, after some transformations on \eqref{eq:mean-v}, it is possible to obtain the classical mean value property. These transformations consist in time reversion and integration with respect to a particular probability distribution function with compact support. These main ideas appear implicit in the proof of Proposition \ref{prop:regul}, the details being left to the reader.
\begin{proof}
Let $a\in\mathbb{R}_+$ be defined by $a=\alpha^{d/2}\Gamma(d/2)2^{d/2-1}$ and let us consider the associated function 
\[
\psi_{a,d}(t):=\sqrt{2t\log\left(\frac{a}{\Gamma(d/2)t^{d/2}2^{d/2-1}}\right)}.
\]
We introduce $(W_t,\ t\ge 0)$ a standard $d$-dimensional Brownian motion and define by $\tau_{a,d}$ the following hitting time
\[
\tau_{a,d}=\inf\{t\ge 0:\ \Vert W_t\Vert=\psi_{a,d}(t) \}.
\]
Let us just notice that this hitting time is bounded by $\alpha=\left(\frac{a}{\Gamma(d/2)2^{d/2-1}}\right)^{2/d}$ and its distribution function is given by Proposition 5.1 in \cite{deaconu-herrmann-2016}
\begin{equation}\label{eq:densit}
p_{a,d}(t)=\frac{1}{2at}\, \psi_{a,d}^{d}(t),\quad 0\le t\le \alpha.
\end{equation}
Furthermore the exit location $W_{\tau_{a,d}}$ is uniformly distributed on the sphere of radius $\psi_{a,d}(\tau_{a,d})$.
Let us consider $h$ a reverse temperature on $A$.  By It\^o's formula, we obtain
\begin{align*}
h(t+\tau_{a,d},x+\sqrt{2}W_{\tau_{a,d}})&=h(t,x)+\int_0^{\tau_{a,d}}
\partial_t h(t+s,x+\sqrt{2}W_s)\,\dint s \\
&\qquad +\sqrt{2}\int_0^{\tau_{a,d}}\partial_x h(t+s,x+\sqrt{2}W_s)\,\dint W_s\\
&\qquad +\int_0^{\tau_{a,d}} \Delta_xh(t+s,x+\sqrt{2}W_s)\,\dint s. 
\end{align*}
If $a$ is small enough, then $(t+\tau_{a,d},x+\sqrt{2}W_{\tau_{a,d}})\in A$ a.s. Using the fact that $h$ is a reverse temperature in $A$, in particular, the continuity of $\partial_xh$ is known, we can prove that the stochastic integral introduced in the It\^o formula is a martingale. Hence the stopping time theorem leads to
\[
h(t,x)=\mathbb{E}[h(t+\tau_{a,d},x+\sqrt{2}W_{\tau_{a,d}})].
\]  
By \eqref{eq:densit}, we get
\begin{align*}
h(t,x)=\frac{1}{\sigma(\mathbb{S}^d)}\int_0^\alpha\int_{\mathbb{S}^d}h(t+u,x+\sqrt{2}\psi_{a,d}(u)y)
\frac{1}{2au}\, \psi_{a,d}^d(u)\,\dint \sigma(y)\,\dint u.
\end{align*}
We introduce the change of variable $u=\alpha s$, such that $s\in]0,1[$, and observe  that $\psi_{a,d}(\alpha s)=\sqrt{2\alpha}\psi_d(s)$ where $\psi_d$ is defined by \eqref{eq:psi}. We get
\begin{align*}
h(t,x)&=\frac{1}{\sigma(\mathbb{S}^d)}\int_0^1\int_{\mathbb{S}^d}h(t+\alpha s,x+\sqrt{2}\psi_{a,d}(\alpha s)y)
\frac{1}{2as}\, \psi_{a,d}^d(\alpha s)\,\dint \sigma(y)\,\dint s\\
&=\frac{1}{\sigma(\mathbb{S}^d)}\int_0^1\int_{\mathbb{S}^d}h(t+\alpha s,x+2\sqrt{\alpha}\psi_d(s)y)
\frac{2^{d/2}\alpha^{d/2}}{2as}\, \psi_d^d(s)\,\dint \sigma(y)\,\dint s.
\end{align*}
Using both the explicit expression of $\alpha^{d/2}$ and the classical formula $\sigma(\mathbb{S}^d)=2\pi^{d/2}/\Gamma(d/2)$ leads to \eqref{eq:mean-v}.
\end{proof}
The reverse statement of the preceding result can also be proved. The first step consists in the following
\begin{proposition} \label{prop1} If $h$ satisfies the mean value property \eqref{eq:mean-v} and is a $\mathcal{C}^{1,2}$-function for any $(t,x)\in A$ and $\alpha>0$ such that $[t,t+\alpha]\times \overline{B(x,2\sqrt{\alpha d/(2e)})}\subset A$, then $h$ is a \emph{reverse temperature} in $A$.
\end{proposition}
\begin{proof}
Let us consider the function $H:[0,\sqrt{\alpha}]\to \mathbb{R}$ defined by 
\[
H(r)=h(t+r^2s,x+2r\psi_d(s)y),
\]
for any $(s,y)\in[0,1]\times\mathbb{S}^d$. Using the Taylor expansion, we get
\begin{equation}\label{eq:Taylor}
H(\sqrt{\alpha})=H(0)+H'(0)\sqrt{\alpha}+\frac{\alpha}{2}\, H''(0)+o(\alpha)
\end{equation}
where $o(\alpha)$ is uniform with respect to both $s$ and $y$ variables. The derivatives of $H$ can be computed explicitly and we get:
\[
H(0)=h(t,x),\quad H'(0)=2\psi_d(s)\sum_{j=1}^d \partial_{x_j}h(t,x)y_j,
\]
\[
H''(0)=2s\partial_t h(t,x)+4\psi_d^2(s)\sum_{1\le i,j\le d}\partial^2_{x_ix_j}h(t,x)y_iy_j.
\]
Applying the mean value property to both sides of \eqref{eq:Taylor}, we obtain
\begin{equation}\label{eq:taylor2}
\sum_{j=1}^d\partial_{x_j}h(t,x)A_j^0  + \partial_t h(t,x) A_1+\sum_{1\le i,j\le d} \partial_{x_ix_j}h(t,x) A_{i,j}=o(\alpha),
\end{equation}
where 
\begin{align*}
A_j^0&= \frac{2\sqrt{\alpha}}{\Gamma(d/2)}\iint_{(s,y)\in [0,1]\times\mathbb{S}^d}\frac{1}{s}\,\psi_d^{d+1}(s)y_j\,\dint\sigma(y)\,\dint s,\\ 
A_1&=\frac{\alpha}{\Gamma(d/2)}\iint_{(s,y)\in [0,1]\times\mathbb{S}^d}\psi_d^d(s)\,\dint \sigma(y)\,\dint s,\\
A_{i,j}&=\frac{2\alpha}{\Gamma(d/2)}\iint_{(s,y)\in [0,1]\times\mathbb{S}^d}\frac{1}{s}\, \psi_d^{d+2}(s)y_iy_j \,\dint\sigma(y)\,\dint s.
\end{align*}
By symmetry arguments, we have $A^0_j=0$ and $A_{i,j}=0$ for $i\neq j$. Let $X_d$ be a random variable whose probability distribution function is \[
p_d(t)=\frac{1}{\Gamma(d/2)t}\, \psi_d^d(t)1_{[0,1]}(t). 
\]
Let us just notice that $p_d(t)=\alpha p_{a,d}(\alpha t)$, $p_{a,d}$ being defined by \eqref{eq:densit}. Then $X_d=e^{-G}$ where $G$ is a random variable which has the gamma distribution of parameters $(d+2)/2$ and $2/d$. In particular, $X_d$ has the same distribution as $(U_1\ldots U_{(d+2)/2})^{2/d}$ if $d$ is even (here $(U_i)_{i\in\mathbb{N}}$ is a sequence of standard uniform independent random variables) and 
$X_d$ has the same distribution as $(U_1\ldots U_{\lfloor d+2\rfloor/2})^{2/d}e^{-G^2/d}$ if $d$ is odd (here $G$ is a standard Gaussian r.v. independent of the sequence $(U_i)_i$). Therefore if $d$ is even, we deduce
\[
A_1=\alpha \mathbb{E}[X_d]=\alpha\mathbb{E}[U_1^{2/d}]\mathbb{E}[U_2^{2/d}]\ldots \mathbb{E}[U_{(d+2)/2}^{2/d}]=\alpha \Big( \frac{d}{d+2} \Big)^{(d+2)/d}.
\]
For the odd case,
\[
A_1=\alpha \mathbb{E}[X_d]=\alpha\mathbb{E}[U_1^{2/d}]\mathbb{E}[U_2^{2/d}]\ldots \mathbb{E}[U_{\lfloor d+2\rfloor/2}^{2/d}]\mathbb{E}[e^{-G^2/d}]=\alpha \Big( \frac{d}{d+2} \Big)^{\lfloor d+2\rfloor/d}\mathbb{E}[e^{-G^2/d}].
\]
Let us now compute $A_{i,i}$ for $1\le i \le n$. First we observe that 
\[
\int_{\mathbb{S}^d}y_i^2\dint\sigma(y)=\frac{1}{d}\sum_{j=1}^d\int_{\mathbb{S}^d}y_j^2
\dint\sigma(y)=\frac{1}{d}.
\]
So using a convenient change of variable, we get
\begin{align*}
A_{i,i}&=\frac{2\alpha}{d\Gamma(d/2)} \int_0^1\frac{1}{s}\,\psi_d^{d+2}(s)\,\dint s=\frac{2\alpha(d+2)}{d^2\Gamma(d/2)}\,\Gamma((d+2)/2)\int_0^1 \frac{1}{\Gamma((d+2)/2)} t^{(d+2)/d}\,\frac{1}{t}\psi_{d+2}^{d+2}(t)\,\dint t\\
&= \alpha\frac{d+2}{d}\,\mathbb{E}[X_{d+2}^{(d+2)/d}]=\alpha \left(\frac{d}{d+2}\right)^{(d+2)/d}\quad \mbox{if}\ d\ \mbox{is even},
\end{align*} 
and $A_{i,i}=\alpha \left(\frac{d}{d+2}\right)^{\lfloor d+2\rfloor/d}\mathbb{E}[e^{-G^2/d}]$ if $d$ is odd. So we note that for any $d\in\mathbb{N}^*$, we proved that
\[
A^0_j=0,\quad A_{i,j}=\delta_{ij}A_1,
\]
where $\delta_{ij}$ is the Kronecker's symbol. Equation \eqref{eq:taylor2} leads therefore to \eqref{def:space-time}.
\end{proof}
In order to prove the equivalence between the notion of \emph{reverse temperature} and the mean value (MV) property defined in \eqref{eq:mean-v}, we prove that the MV formula implies the regularity of the solution. This is the case when the dimension of the space is large enough.  Intuitively the regularity increases as the dimension of the space increases.
\begin{proposition}
\label{prop:regul} Let $d>4$ and let $h$ be a bounded function, defined on an open set $A$ and satisfying the mean value property \eqref{eq:mean-v}. Then $h$ is a $\mathcal{C}^{1,2}(A,\mathbb{R})$-function.
\end{proposition}
In the historical proof of the regularity associated to the Laplace operator (Proposition 2.5 in \cite{K-S}), the key argument is to introduce a convolution with respect to a $\mathcal{C}^\infty$-function with compact support. For the heat equation, one needs to handle quite differently and will not be able to prove regularity of infinite order in any case.
\begin{proof}
Let us assume that $h:A\to\mathbb{R}$ is a bounded function satisfying the mean value property for $\alpha$ small enough, smaller than some $\alpha_M>0$. We introduce a $\mathcal{C}^\infty(\mathbb{R}_+,\mathbb{R}_+)$- probability distribution function function $k$ whose support is included in $[\alpha_m,\alpha_M]$ with $0<\alpha_m<\alpha_M$. Rewriting the straightforward equality
\[
\int_0^\infty k(\alpha)h(t,x)\, \dint \alpha=h(t,x)
\]
leads to the following mean value property
\begin{equation}
\label{eq:int-m-v}
h(t,x)=\frac{1}{2\pi^{d/2}}\iiint_{(\alpha,s,y)\in \mathbb{R}_+\times]0,1[\times\mathbb{S}^d} k(\alpha) h(t+\alpha s,x+2\sqrt{\alpha}\psi_d(s)y)\,\frac{\psi_d^d(s)}{s}\,\dint \sigma(y)\,\dint s\,\dint \alpha.
\end{equation}
In order to do calculation, it is more convenient in this situation to use spherical coordinates in $\mathbb{S}^d$; for $y\in\mathbb{S}^d$ we define $(\theta_1,\ldots,\theta_{d-1})$ by
\begin{align*}
y_1=\cos(\theta_1),\quad
y_k=\prod_{j=1}^{k-1}\sin(\theta_j)\cos(\theta_{k}), \ \mbox{for}\ 2\le k\le d,
\end{align*}
where $\theta_1,\ldots,\theta_{d-2}\in [0,\pi]$, $\theta_{d-1}\in [0,2\pi]$ and $\theta_d=0$. The change of measure is therefore given by
\[
\dint \sigma(y)\longleftrightarrow \sin^{d-2}\theta_1\ldots \sin^{2}(\theta_{d-3})\sin(\theta_{d-2}) \dint\theta_1\ldots  \dint\theta_{d-1}.
\]
Let us now consider another system of coordinates $(u,z_1,\ldots,z_d)$ replacing $(\alpha, s, \theta_1,\ldots,\theta_{d-1})$ and given by:
\begin{align*}
u=\alpha s,\quad z_1=2\sqrt{\alpha}\psi_d(s)\cos(\theta_1),\quad z_k=2\sqrt{\alpha}\psi_d(s)\prod_{j=1}^{k-1}\sin(\theta_j)\cos(\theta_{k}), \ \mbox{for}\ 2\le k\le d,
\end{align*}
we recall that $\theta_d=0$ by convention. Let us just observe that
\[
\Vert z\Vert ^2=4\alpha \psi_d^2(s)=4\alpha s\log(s^{-d/2})=4u\log(s^{-d/2}).
\]
Hence $s=\exp\left(-\frac{\Vert z\Vert^2}{2du}\right)$ and $\alpha=u\exp\frac{\Vert z\Vert^2}{2du}$. The Jacobian determinant of the change of variable is equal to:
\begin{align*}
{\rm Jac}&:=\frac{\partial (u,z_1,z_2,\ldots,z_d)}{\partial (\alpha,s,\theta_1,\ldots,\theta_{d-1})}\\
&=(2s\sqrt{\alpha}\psi_d'(s)-\sqrt{\alpha}\psi_d(s))(2\sqrt{\alpha}\psi_d(s))^{d-1}\sin^{d-2}\theta_1\ldots \sin^{2}(\theta_{d-3})\sin(\theta_{d-2})\\
&=-d\,2^{d-2}\alpha^{d/2}s\psi_d(s)^{d-2}\sin^{d-2}\theta_1\ldots \sin^{2}(\theta_{d-3})\sin(\theta_{d-2}).
\end{align*}
The mean value property therefore becomes:
\begin{equation}
\label{eq:int-m-v1}
h(t,x)=\frac{1}{2\pi^{d/2}}\iint_{(u,z)\in \mathbb{R}_+\times\mathbb{R}^d} k(ue^{\frac{\Vert z\Vert^2}{2du}}) h(t+u,x+z)\,\frac{\Vert z\Vert^2 e^{(2-d)\frac{\Vert z\Vert^2}{4du}}}{d\,2^d u^{d/2+1}}\,\dint z\,\dint u.
\end{equation}
Let us just note that, for the particular choice $k(x)=\frac{d}{2}\ r^{-d/2}\  x^{d/2-1}1_{\{\Vert x\Vert\le r^2/(4\pi)\}}$ with $r>0$, we obtain the classical mean value formula presented in the statement of Theorem 3 p. 53 in \cite{evans-2010}. Of course in this case the condition on the smoothness of $k$ is not satisfied.

Since $\alpha>0$ if and only if $u>0$ and since the support of $k$ belongs to $\mathbb{R}_+^*$, we can replace the integration domain $\mathbb{R}_+\times \mathbb{R}^d$ by $\mathbb{R}^{d+1}$.  The formula \eqref{eq:int-m-v1} can be written as the following convolution integral
\begin{align}
\label{eq:meanff}
h(t,x)=\int_{\mathbb{R}^{d+1}}h(u,z)w(u-t,z-x)\dint z\,\dint u,\quad (t,x)\in A,
\end{align}
where 
\begin{equation}
\label{eq:def-w}
w(t,x):=\frac{k(te^{\frac{\Vert x\Vert^2}{2dt}})}{2\pi^{d/2}}\,\frac{\Vert x\Vert^2 e^{(2-d)\frac{\Vert x\Vert^2}{4dt}}}{d\,2^d t^{d/2+1}},\quad (t,x)\in\mathbb{R}^{d+1}.
\end{equation}
In fact, the support of the function $w$ is compact due to $k$ whose support belongs to $[\alpha_m,\alpha_M]$. The aim is now to deduce the regularity property of $h$ from that of $w$.\\ \\
Step 1. \emph{Regularity with respect to the time variable.}\\
For any $(t_0,x_0)\in A$, we choose a small neighborhood of the form $]t_0-\eta,t_0+\eta[\times B(x_0,\eta)$ which is contained in $A$ and we are going to prove that $h$ is $\mathcal{C}^1$ with respect to the time variable in this neighborhood.\\
Let us denote by $\alpha:=(u-t)e^{\frac{\Vert z-x\Vert^2}{2d(u-t)}}$ then 
\begin{equation}\label{eq:derivalfa}
\frac{\partial \alpha}{\partial t}= -\Big(1-\frac{\Vert z-x\Vert^2}{2d (u-t)}\Big)\frac{\alpha}{(u-t)}=-\Big(1-\frac{\Vert z- x\Vert^2}{2d (u-t)}\Big)e^{\frac{\Vert z- x\Vert^2}{2d(u-t)}}.
\end{equation}
Introducing $g:\mathbb{R}\to \mathbb{R}_+$ defined by $g(x)=(d 2^{d+1}\pi^{d/2})^{-1}\ k(x)x^{1-d/2}$ which is a $\mathcal{C}^\infty$ function with compact support due to the regularity of the function $k$ and since the support of $k$ does not contain a small neighborhood of the origin, we obtain
\(
w(u-t,z-x)=g(\alpha)\ \frac{\Vert z-x\Vert^2}{(u-t)^2}.
\)
Hence
\begin{align}\label{eq:derivtemps}
\frac{\partial w}{\partial t}(u-t,z-x)&= \Vert z-x\Vert^2 \left(\frac{g'(\alpha)}{(u-t)^2}\frac{\partial \alpha}{\partial t}+\frac{2 g(\alpha)}{(u-t)^3}\right)\nonumber\\
&=\Vert z-x\Vert^2\left( -\frac{g'(\alpha)}{(u-t)^2}\Big(1-\frac{\Vert z- x\Vert^2}{2d (u-t)}\Big)\frac{\alpha}{(u-t)}+\frac{2 g(\alpha)}{(u-t)^3}\right).
\end{align}
Let us fix $x\in\mathbb{R}^d$. We can observe that both $(t,u,z)\mapsto w(u-t,z-x)$ and $(t,u,z)\mapsto \frac{\partial w}{\partial t}(u-t,z-x)$ are continuous for $t\neq u$ that is for $(u,z)$ belonging to $\mathbb{R}\times\mathbb{R}^d\setminus(\{t\}\times\mathbb{R}^d)$ (the complementary set is negligible for the Lebesgue measure). Moreover, due to the compact support of $k$, for any $\varepsilon>0$ small enough, there exists a constant $\kappa_\varepsilon>0$ such that 
\begin{equation}\label{eq:ineg}
\Vert z-x\Vert ^2\le 2d(u-t)\log(\alpha_M/(u-t))\le \kappa_\varepsilon (u-t)^{1-\varepsilon}. 
\end{equation}
Hence \eqref{eq:derivtemps} implies the existence of  $C_0$, $C_1$ and $C_2$ independent of $z$, $u$ and $t$ such that
\begin{equation}\label{eq:majorderivt}
\left|\frac{\partial w}{\partial t}(u-t,z-x)\right|\le C_0\left(\frac{1}{\Vert z-x\Vert^{6/(1-\varepsilon)-2}}+\frac{1}{\Vert z-x\Vert^{8/(1-\varepsilon)-4}}\right)1_{\{0\le u\le C_1\}}1_{\{\Vert z-x\Vert\le  C_2\}}
\end{equation}
for any $(t,x)$ in $]t_0-\eta,t_0+\eta[\times B(x_0,\eta)\subset A$. The right hand side of \eqref{eq:majorderivt} is integrable as soon as the space dimension satisfies $d>4$. Since $h$ is a bounded function the Lebesgue theorem permits to apply results involving differentiation under the integral sign. The function $h$ is therefore $\mathcal{C}^1$ with respect to the time variable.\\ \\
Step 2. \emph{Regularity with respect to the space variable.}\\
The computation is quite similar as in the first step. We have:
\[
\frac{\partial w}{\partial x_i}(u-t,z-x)=-\frac{(z_i-x_i)}{(u-t)^2}\Big(\frac{g'(\alpha)\alpha \Vert z-x\Vert^2}{d(u-t)}+2g(\alpha)\Big),\quad 1\le i \le d,
\]
and
\begin{align*}
\frac{\partial^2 w}{\partial x_i\partial x_j}(u-t,z-x)&=\frac{\alpha(z_i-x_i)(z_j-x_j)}{d(u-t)^3}\left\{\frac{g''(\alpha)\alpha+g'(\alpha)}{d(u-t)}\,\Vert z-x\Vert^2+4 g'(\alpha)\right\}\\
&\qquad +\frac{1}{(u-t)^2}\left\{ \frac{g'(\alpha)\alpha }{d(u-t)}\,\Vert z-x\Vert^2+2g(\alpha) \right\}\delta_{ij}.
\end{align*}
The second derivative is a continuous function on  $\mathbb{R}^{d+1}\setminus \{u=t\}$.
Since $g$ is $\mathcal{C}^\infty$ with a compact support, we know that $\alpha$ is bounded and therefore, using \eqref{eq:ineg}, we obtain the following bound:
\begin{align*}
\left|\frac{\partial^2 w}{\partial x_i\partial x_j}(u-t,z-x)\right| \le C_0\left(\frac{\Vert z-x\Vert ^4 }{(u-t)^4}+\frac{\Vert z-x\Vert ^2}{(u-t)^3} +\frac{1}{(u-t)^2 }\right)1_{\{0\le u\le C_1\}}1_{\{\Vert z-x\Vert\le  C_2\}}\\
\le C_0\left(\frac{1}{\Vert z-x\Vert^{8/(1-\varepsilon)-4}}+\frac{1}{\Vert z-x\Vert^{6/(1-\varepsilon)-2}}
+\frac{1}{\Vert z-x\Vert^{4/(1-\varepsilon)}}\right)1_{\{0\le u\le C_1\}}1_{\{\Vert z-x\Vert\le  C_2\}}.
\end{align*}
(Here $C_0$ is just a generic constant: the values can change from one computation line to the following one). The conclusion is based on the same argument presented in Step 1: the boundedness of $h$ permits to conclude that $h$ is $\mathcal{C}^2$ with respect to the space variable as soon as $d>4$.
\end{proof}
Let us note that the heat equation and the Laplace equation have some similar properties. In particular, solutions of these equations are automatically smooth. More precisely, if $u\in\mathcal{C}^{1,2}(A,\mathbb{R})$ and is a reverse temperature then  $u\in\mathcal{C}^{\infty}(A\times\mathbb{R}_+,\mathbb{R})$ (see for instance Theorem 8 page 59 in \citep{evans-2010}). In fact, as soon as the dimension $d$ is large enough ($d>4$), the mean value property implies the smoothness as an immediate consequence of Proposition \ref{prop:regul}.

All results presented so far in this section have an important advantage, they can be adapted to other situations for instance by looking backward in time, or equivalently time reverting.
This observation permits to study properties of the heat equation.
\begin{definition}
A function $h:A\mapsto \mathbb{R}$ is said to be a \emph{temperature} in $A$ if $h$ is a $\mathcal{C}^{1,2}$-function satisfying the heat equation:
\begin{equation}
\label{def:space-time-rev}
\partial_t h(t,x)-\Delta_x h(t,x)=0,\quad \forall (t,x)\in A.
\end{equation}
\end{definition}
By Proposition \ref{prop:me}, Proposition \ref{prop1} and Proposition \ref{prop:regul}, we obtain the following:
\begin{thm}\label{thm1}
\begin{enumerate}
\item Let $A\subset \mathbb{R}_+\times \mathbb{R}^d$ be a non empty open set. If a function $h$ is a \emph{temperature} in $A$ then it has the following mean value property:
\begin{equation}
\label{eq:mean-v-r}
h(t,x)=\frac{1}{2\pi^{d/2}}\iint_{(s,y)\in ]0,1[\times\mathbb{S}^d}\frac{1}{s}\, h(t-\alpha s,x+2\sqrt{\alpha}\psi_d(s)y)\,\psi_d^d(s)\,\dint \sigma(y)\,\dint s,
\end{equation}
where $\mathbb{S}^d$ is the $d$-dimensional sphere of radius $1$, $\sigma$ is the Lebesgue measure on $\mathbb{S}^d$ and $\psi_d$ is defined in \eqref{eq:psi}.
Equation \eqref{eq:mean-v-r} is satisfied for any $\alpha>0$ such that $[t-\alpha,t]\times \overline{B(x,2\sqrt{\alpha d/(2e)})}\subset A$.
\item If $h$ satisfies the mean value property \eqref{eq:mean-v-r} and is a $\mathcal{C}^{1,2}$-function for any $(t,x)\in A$ and $\alpha>0$ such that $[t-\alpha,t]\times \overline{B(x,2\sqrt{\alpha d/(2e)})}\subset A$ then $h$ is a \emph{temperature} in $A$.
\item For $d>4$, any bounded function $h$ is a temperature iff it satisfies the mean value property \eqref{eq:mean-v-r}.
\end{enumerate}
\end{thm}
\section{Solving the Initial-Boundary Value Problem}
\label{sec:solving}
This section deals with existence and uniqueness for solutions of the Initial-Boundary Value Problem \eqref{eq:DP} in a bounded domain $\mathcal{D}$.  These results are deeply related to the existence of a particular time-discrete martingale: we define $M_n:=(T_n,X_n)$ a sequence of $\mathbb{R}_+\times \mathcal{D}$-valued random variables. In order to define this sequence we introduce $\delta(x,\partial\mathcal{D})$ the Euclidean distance between the point $x$ and the boundary of the domain. We also introduce the function $\alpha$ given by:
\begin{equation}\label{eq:def:alp}
\alpha(u,v)=\min\Big(u,\frac{e}{2d}\,\delta^2(v,\partial \mathcal{D}) \Big).
\end{equation}

Let us consider:
\begin{itemize}
\item  $(U_n)_{n\ge 1}$ a sequence of independent uniformly distributed random vectors on $[0,1]^{\lfloor d/2\rfloor+1}$. We denote by $\Pi^U_n$ the product of all coordinates of $U_n$. 
\item $(G_n)_{n\ge 1}$ a sequence of independent standard Gaussian r.v. 
\item $(V_n)_{n\ge 1}$ a sequence of independent uniformly distributed random vectors on the unit sphere of dimension $d$, centered on the origin. 
\end{itemize}
Further, we assume that these three sequences are independent. We define by $\mathcal{F}_n$ the natural filtration generated by the sequences $(U_n)$, $(G_n)$ and $(V_n)$. Let $\mathcal{F}_0$  note the trivial $\sigma$-algebra. Let us introduce:
\[
R_{n+1}:=\Big(\Pi^U_{n+1}\Big)^{2/d}\exp\Big\{-\left( 1-\frac{2}{d}\lfloor \frac{d}{2}\rfloor\right)G_{n+1}^2\Big\}.
\]
The initial value of the sequence $(T_n,X_n)$ is then  $(T_0,X_0)=(t,x)$ and the sequence is defined by recurrence as follows: for $n\ge 0$,
\begin{align}\label{eq:algo}
\left\{\begin{array}{l}
T_{n+1}=T_n-\alpha(T_n,X_n) R_{n+1},\\[5pt]
X_{n+1}=X_n+2\sqrt{\alpha(T_n,X_n)}\psi_d(R_{n+1})V_{n+1}.
\end{array}\right.
\end{align}
Let us first note that, due to the definition, the sequence $(T_n,X_n)$ belongs always to the closed set $[0,t]\times \overline{\mathcal{D}}$: the sequence is therefore bounded. Moreover as soon as $M_n$ reaches the boundary of $[0,t]\times \overline{\mathcal{D}}$ its value is frozen.
\begin{lemma} \label{lem:matg} If $h$  belongs to  $\mathcal{C}^{1,2}([0,t]\times \overline{\mathcal{D}})$ and if it is a temperature in $[0,t]\times \mathcal{D}$, then $\mathcal{M}_n:=h(T_n,X_n)$ is a bounded $\mathcal{F}$-martingale.
\end{lemma}
\begin{proof}
Since $h$ is a continuous function on a compact set, it is bounded. Therefore the stochastic process $\mathcal{M}_n$ itself is bounded. We obtain
\begin{align*}
\mathbb{E}[\mathcal{M}_{n+1}|\mathcal{F}_n]&=\mathbb{E}[h(T_{n+1},X_{n+1})|\mathcal{F}_n]=:H(T_n,X_n),
\end{align*}
where 
\begin{align*}
H(u,v)=\mathbb{E}\Big[ h\Big(u-\alpha(u,v)R_{n+1},v+2\sqrt{\alpha(u,v)}\psi_d(R_{n+1})V_{n+1}\Big) \Big].
\end{align*}
Since the pdf of $R_{n+1}$ is given by $f_R(s)=\frac{1}{\Gamma(d/2)}\frac{\psi_d^d(s)}{s}1_{[0,1]}(s)$ and since $V_{n+1}$ is uniformly distributed on the sphere, we obtain:
\begin{align*}
H(u,v)=\frac{1}{\Gamma(d/2)\sigma(\mathbb{S}^d)}\iint_{(s,y)\in ]0,1[\times\mathbb{S}^d}\frac{1}{s}\, h(t-\alpha(u,v) s,x+2\sqrt{\alpha(u,v)}\psi_d(s)y)\,\psi_d^d(s)\,\dint \sigma(y)\,\dint s.
\end{align*}
If $h$ belongs to  $\mathcal{C}^{1,2}([0,t]\times \overline{\mathcal{D}})$ and if it is a temperature in $[0,t]\times \mathcal{D}$, then Theorem \ref{thm1} implies the mean value property. Hence $H(u,v)=h(u,v)$. We deduce easily that 
\[
\mathbb{E}[\mathcal{M}_{n+1}|\mathcal{F}_n]=h(T_n,X_n)=\mathcal{M}_n\quad a.s.
\]
\end{proof}
\begin{lemma}\label{lem:cv}
The process $M_n=(T_n,X_n)$ converges almost surely as $n\to \infty$ to a limit $(T_\infty,X_\infty)$ that belongs to the set $\{0\}\times \overline{\mathcal{D}}\ \cup\  ]0,t[\times \partial \mathcal{D}$.
\end{lemma}
\begin{proof}
Let us consider the function $h(t,x)=x_i$ the $i$-th coordinate of $x\in\mathbb{R}^d$. We observe that $h$ is a temperature and belongs to $\mathcal{C}^{1,2}(\mathbb{R}_+,\mathbb{R}^d)$. By Lemma \ref{lem:matg}, we deduce that $\mathcal{M}_n:=h(T_n,X_n)=X_n(i)$, the $i$-th coordinate of $X_n$, is a bounded martingale therefore it converges a.s. towards $X_\infty(i)$. Since all coordinates converge we deduce that $X_n\to X_\infty$ a.s.\\ 
 Moreover since $T_n$ is a non-increasing sequence of non negative random times, it converges a.s. towards a r.v. $T_\infty$ which belongs to $[0,t]$. 
The sequence $(T_n,X_n)$ belongs to the closed set $[0,t]\times\overline{\mathcal{D}}$, consequently its limit belongs to the same set.\\
Let us change the starting point of the Markov chain by replacing $(t,x)$ by $(T_\infty,X_\infty)$ we obtain a constant Markov chain $M_n=(T_\infty,X_\infty)$ for all $n\ge 0$. Let us assume that this limit does not belong to $\{0\}\times \overline{\mathcal{D}}\ \cup\  ]0,t[\times \partial \mathcal{D}$ then $\alpha(T_\infty,X_\infty)>0$ a.s. We deduce that $M_1\neq M_0$ a.s. for this new Markov chain since  their first coordinates are different a.s. This fact obviously cannot be satisfied by a constant Markov chain, therefore  $(T_\infty,X_\infty)\in \{0\}\times \overline{\mathcal{D}}\ \cup\  ]0,t[\times \partial \mathcal{D}$.
\end{proof}
\begin{proposition}\label{prop:uniq}
\emph{(uniqueness)} Set $T>0$. Let $u$ be a $\mathcal{C}^{1,2}([0,T]\times\mathcal{D})$-function satisfying the Initial-Boundary Value Problem  \eqref{eq:DP} and continuous with respect to both variables on $[0,T]\times \overline{\mathcal{D}}$. Then $u$ is unique and given by the expression
\begin{equation}\label{eq:prop:uniq}
u(t,x)=\mathbb{E}_{(t,x)}[f(T_\infty,X_\infty)1_{\{ X_\infty\in\partial \mathcal{D} \}}]+\mathbb{E}_{(t,x)}[f_0(X_\infty)1_{\{  X_\infty\notin\partial \mathcal{D} \}}], \quad \mbox{for}\quad (t, x)\in[0,T]\times\overline{\mathcal{D}}.
\end{equation}
\end{proposition}
\begin{proof}
By Lemma \ref{lem:matg}, the process $\mathcal{M}_n$ is a bounded martingale. Moreover Lemma \ref{lem:cv} implies that $(T_n,X_n)$ converges to $(T_\infty,X_\infty)$. Since $u$ is a continuous function, we deduce that $\mathcal{M}_n$ converges a.s. and in $L^2$ towards $u(T_\infty,X_\infty)$. In particular, the martingale property leads to
\[
u(t,x)=\mathbb{E}[u(T_\infty,X_\infty)].
\]
In order to conclude it suffices to use the initial and boundary conditions. Indeed Lemma \ref{lem:cv} ensures that $(T_\infty,X_\infty)$ belongs to the set $\{0\}\times \overline{\mathcal{D}}\ \cup\  ]0,t[\times \partial \mathcal{D}$.
\end{proof}

We refer to Friedman \cite{Friedman} for the existence of a solution to the Initial-Boundary Value Problem \eqref{eq:DP}. More precisely, if the following particular conditions are fulfilled:
\begin{itemize} 
\item  $f$ and $f_0$ are continuous functions such that $f(0,x)=f_0(x)$, 
\item the domain has an outside strong sphere property,
\end{itemize}
then there exists a smooth solution $u$ to  \eqref{eq:DP}: $u\in\mathcal{C}^\infty( \mathbb{R}_+\times \mathcal{D},\mathbb{R})$. This statement results from a combination of Theorem 9 page 69 and Corollary 2 page 74 in \citep{Friedman}.
%
%
%

\section{Approximation of the solution for an Initial-Boundary Value Problem}\label{sec:approx}
The aim of this section is to contruct an algorithm which approximates $u(t,x)$, the solution of an Initial-Boundary Value Problem when $(t,x)$ is given. For the Dirichlet problem such an algorithm was introduced by M\"uller \cite{muller_56} and is called the \emph{Random Walk on spheres}. We are concerned with the heat equation instead of the Laplace equation and therefore propose an adaptation of this algorithm in order to consider also the time variable. The algorithm is based on the sequence of random variables $M_n=(T_n,X_n)$ defined by \eqref{eq:algo}. \\
We introduce a stopping rule: let $\varepsilon>0$ be a small parameter, we define $\mathcal{N}_\varepsilon$ the stopping time:
\begin{equation}
\label{def:neps}
\mathcal{N}_\varepsilon:=\inf\{n\ge 0:\ \alpha(T_n,X_n)\le \varepsilon\},
\end{equation}
where $\alpha$ is given by \eqref{eq:def:alp}. 
\begin{enumerate}
\item If $\delta^2(X_{\mathcal{N}_\varepsilon},\partial \mathcal{D})\le \frac{2\varepsilon d}{e}$ then we choose $X_\varepsilon\in \partial \mathcal{D}$ such that 
\[
\delta(X_{\mathcal{N}_\varepsilon}, X_\varepsilon )=\delta(X_{\mathcal{N}_\varepsilon},\partial \mathcal{D})
\]
and we denote by $T_\varepsilon:=T_{\mathcal{N}_\varepsilon}$.
\item If $\delta^2(X_{\mathcal{N}_\varepsilon},\partial \mathcal{D})>\frac{2\varepsilon d}{e}$ then we set $T_\varepsilon=0$ and $X_\varepsilon:=X_{\mathcal{N}_\varepsilon}$.
\end{enumerate}
We are now able to give an approximation of the solution to \eqref{eq:DP} by using the definition:
\begin{equation}\label{eq:approx}
u^\varepsilon(t,x)=\mathbb{E}_{(t,x)}[f(T_\varepsilon,X_\varepsilon)1_{\{ X_\varepsilon\in\partial \mathcal{D} \}}]+\mathbb{E}_{(t,x)}[f_0(X_\varepsilon)1_{\{  X_\varepsilon\notin\partial \mathcal{D} \}}], \quad \mbox{for}\quad (t, x)\in[0,T]\times\overline{\mathcal{D}}.
\end{equation}
\begin{proposition}
\label{thm:approx} Let us assume that the Initial-boundary Value Problem \eqref{eq:DP} admits an unique $\mathcal{C}^{1,2}([0,T]\times\mathcal{D})$-solution $u$, defined by \eqref{eq:prop:uniq}. We introduce the approximation $u^\varepsilon$ given by \eqref{eq:approx}. Then $u^\varepsilon$ converges towards $u$, as $\varepsilon\to 0$, uniformly with respect to $(t,x)$. Moreover there exist $\kappa_{T,\mathcal{D}}(u)>0$ and $\varepsilon>0$ such that
\[
|u(t,x)-u^\varepsilon(t,x)|\le \kappa_{T,\mathcal{D}}(u)\sqrt{\varepsilon},\quad \forall \varepsilon\le \varepsilon_0,\ (t,x)\in [0,T]\times\mathcal{D}.
\]
\end{proposition}
\begin{proof}
Using the definition of $u$ (resp. $u^\varepsilon$) in \eqref{eq:prop:uniq} (resp. \eqref{eq:approx}), we obtain:
\begin{align*}
|u(t,x)-u^\varepsilon(t,x)|=|\mathbb{E}[ u(T_\infty,X_\infty)]-\mathbb{E}[u(T_\varepsilon, X_\varepsilon)]|.
\end{align*}
Since $n\mapsto u(T_n,X_n)$ is a bounded martingale and since $\mathcal{N}_\varepsilon$ is a finite stopping time, we can apply the optimal stopping theorem leading to
\begin{align*}
|u(t,x)-u^\varepsilon(t,x)|=|\mathbb{E}[ u(T_{\mathcal{N}_\varepsilon},X_{\mathcal{N}_\varepsilon})]-\mathbb{E}[u(T_\varepsilon, X_\varepsilon)]|\le \hat{\kappa}_{T,\mathcal{D}}(u) \mathbb{E}\Big[\max(|T_{\mathcal{N}_\varepsilon}-T_\varepsilon|, |X_{\mathcal{N}_\varepsilon}-X_\varepsilon|)\Big]
\end{align*}
where 
\begin{align*}
\hat{\kappa}_{T,\mathcal{D}}(u):=\sup_{(t,x)\in [0,T]\times\mathcal{D}} \max\Big\{\Big| \frac{\partial u}{\partial t}(t,x) \Big|,\Big|\frac{\partial u}{\partial x}(t,x)\Big|\Big\}.
\end{align*}
Taking into account the two different situations $\delta^2(X_{\mathcal{N}_\varepsilon},\partial \mathcal{D})>\frac{2\varepsilon d}{e}$ or $\delta^2(X_{\mathcal{N}_\varepsilon},\partial \mathcal{D})\le \frac{2\varepsilon d}{e}$, we deduce that 
\[
\max(|T_{\mathcal{N}_\varepsilon}-T_\varepsilon|, |X_{\mathcal{N}_\varepsilon}-X_\varepsilon|)\le \max\Big(\varepsilon, \sqrt{\frac{2\varepsilon d}{e}}\Big).
\]
The statement follows with the particular choice $\kappa_{T,\mathcal{D}}(u)=\hat{\kappa}_{T,\mathcal{D}}(u)  \sqrt{\frac{2 d}{e}}$.
\end{proof}
 Let us now focus our attention on the number of steps needed by the algorithm \eqref{eq:algo} before stopping. In order to present the main result, we need some particular properties on the domain $\mathcal{D}$.\\
In the sequel, we shall assume that $\mathcal{D}$ is a $0$-thick domain, that is: 
there exists a constant $C>0$ (so-called the thickness of the domain) such that
\begin{equation}\label{eq:thick}
H^d(B(x,r)\setminus\mathcal{D})\ge C r^d,\quad \forall r<1,\ \forall x\in \partial \mathcal{D}.
\end{equation}
Here $H^d(S)$ denotes the $d$-dimensional Hausdorff content of the set $S$. This property is namely satisfied by
\begin{itemize}
\item convex domains;
\item domains satisfying a cone condition;
\item bounded domains with a smooth boundary $\partial \mathcal{D}$.
\end{itemize}
We observe that the assumption is quite weak. For such domains, we can prove the following rate of convergence.

\begin{thm}
\label{thm:mean-number}
Let $\mathcal{D}\subset B(0,1)$ be a $0$-thick domain.
The number of steps $\mathcal{N}_\varepsilon$, of the approximation algorithm, is almost surely finite. Moreover there exist constants $C>0$ and $\varepsilon_0>0$ such that
\begin{equation}
\label{eq:mean-number}
\mathbb{E}[\mathcal{N}_\varepsilon]\le C|\log \varepsilon|,\quad \mbox{for all }\ \varepsilon\le \varepsilon_0.
\end{equation}
\end{thm}
 The proof of this result is an adaptation of the classical random walk on spheres \cite{Binder-Bravermann}. Nevertheless the dynamics of both coordinates of the random walk on spheres $(T_n,X_n)_{n\ge 0}$ being definitively different, this adaptation requires a quite tedious effort. In particular, we need to introduce a particular submartingale, based on the random walk, whose properties permit to prove the rate of convergence.
\subsection{Submartingale related to the Riesz potential}

We consider in this section the $0$-thick domain $\mathcal{D}$ which is included in the unit ball of $\mathbb{R}^d$ (assumption of Theorem \ref{thm:mean-number}). We introduce the set $\mathcal{M}$ of all Borel measures $\mu$ supported inside $\overline{B(0,2)}$ and outside of $\mathcal{D}$, satisfying the following condition:
\begin{equation}
\label{eq:cond:mu}
\mu(B(x,r))\le r^d,\quad \forall x\in\mathbb{R}^d,\quad \forall r>0.
\end{equation}
Let us define the so-called \emph{energy function} $U$:
\[
U(x)=\sup_{\mu\in\mathcal{M}}U^\mu(x),
\]
where $U^\mu$ stands for the Riesz potential of the measure $\mu$, that is,
\begin{equation}
U^\mu(x)=\int_0^\infty \frac{\mu(B(x,r))}{r^{d+1}}\,dr.
\end{equation}
The definition of $U$ obviously implies that $U(x)\ge 0$ for any $x\in\mathbb{R}^d$.
\begin{rem}\label{rem} Binder and Bravermann  \cite{Binder-Bravermann} gave several properties of this energy function. We just recall some of them:
\begin{enumerate}
\item Since the set of measures $\mathcal{M}$ is weakly$^*$-compact, 
%
%
%
%
%
%
%
%
%
%
there exists a family of measures $\mu_x$, belonging to $\mathcal{M}$, such that $U(x)=U^{\mu_x}(x)$. This property will play a crucial role in the proof of Proposition \ref{prop:classic}.
\item The energy function $U$ is subharmonic in $\mathcal{D}$. Consequently, due to the construction of the random walk $(T_n,X_n)$ which is based on uniform random variables on moving spheres, the process $(U(X_n))_{n\ge 0}$ is a submartingale with respect to the filtration generated by $(T_n,X_n)_{n\ge 0}$: 
\[
\mathcal{F}_n:=\sigma\{T_1,\ldots,T_n,X_1\ldots, X_n\}.
\]
Hence 
\begin{equation}\label{eq:mtg}
\mathbb{E}[U(X_{n+1})|\mathcal{F}_n]\ge U(X_n)\quad a.s.
\end{equation}
\item Easy computations on the Riesz potential permits to prove that
\begin{equation}
\label{eq:riesz-sup}
U(x)\le \log\frac{3}{\delta(x,\partial \mathcal{D})}+\frac{1}{d},\quad \forall x\in\mathcal{D}.
\end{equation}
In particular, we obtain an important property of the submartingale $(U(X_n))_{n\ge 0}$: 
\begin{equation}\label{eq:riesz-sup-2}
{\mbox { If }} U(X_n)\ge \log\frac{3}{\varepsilon}+\frac{1}{d}\quad\mbox{than}\quad \delta(X_n,\mathcal{D})\le \varepsilon.
\end{equation}
\end{enumerate}
\end{rem}
These properties, of the energy function $U$, permit to sketch the proof of the convergence in the classical random walk on spheres case. Indeed we know that $U(X_n)$ is a submartingale and the algorithm stops before $U(X_n)$ becomes too large. So, it suffices to focus the attention on the time needed by the submartingale to exceed some given large threshold. 

In the algorithm described by \eqref{eq:algo}, a large value of $U(X_n)$ is not sufficient to ensure that the stopping rule has been reached. Indeed the stopping procedure depends on both space and time variables, through the condition: $\alpha(T_n,X_n)\le \varepsilon$. That's why we need to adapt the classical study by considering a martingale based on the Riesz potential but taking also into account the decreasing time sequence $(T_n)_{n\ge 0}$.\\[5pt]
Let us define the \emph{modified energy function} on $\mathbb{R}_+\times \mathbb{R}^d$ by
\begin{equation}\label{eq:def:U}
\mathbb{U}(t,x)=\max\Big\{\frac{1}{2} \log\left(\frac{3}{t}\right),U(x)\Big\}.
\end{equation}
This function will play a similar role as the energy function (in the classical case). In particular, if we apply $\mathbb{U}$ to the sequence  $(T_n,X_n)$, we obtain a submartingale with nice properties.
\begin{lemma}\label{lem:submart} We define $\mathbb{U}_n:=\mathbb{U}(T_n,X_n)$. Then the process $(\mathbb{U}_n)_{n\ge 0}$ is a $\mathcal{F}$-submartingale.
\end{lemma}
\begin{proof} 
First we can notice that
\[
T_{n+k}\le T_n \Rightarrow \log\frac{3}{T_{n+k}}\ge \log\frac{3}{T_{n}}.
\]
By Jensen's inequality (see Lemma \ref{lem:appen:Jensen}) and using the submartingale property \eqref{eq:mtg} of $U(X_n)$, we obtain
\begin{align*}
\mathbb{E}\Big[\mathbb{U}_{n+1}\Big\vert \sigma(T_n,X_n)\Big]&=\mathbb{E}\Big[\max\Big\{\frac{1}{2} \log\frac{3}{T_{n+1}},U(X_{n+1})\Big\}\Big\vert \sigma(T_n,X_n)\Big]\\
&\ge \mathbb{E}\Big[\max\Big\{ \frac{1}{2}\log\frac{3}{T_{n}},U(X_{n+1})\Big\}\Big\vert \sigma(T_n,X_n)\Big]\\
&\ge \max\Big\{\frac{1}{2} \log\frac{3}{T_{n}},\mathbb{E}[U(X_{n+1})\vert \sigma(T_n,X_n)]\Big\}\\
&\ge \max\Big\{ \frac{1}{2}\log\frac{3}{T_{n}},U(X_n)\Big\}=\mathbb{U}_n.
\end{align*}
We deduce that $(\mathbb{U}_n)$ is a submartingale as announced.
\end{proof}
In order to describe an upper-bound for the sequence $\mathbb{E}[\mathbb{U}_n^2]$, we first point out an inequality relating $\mathbb{U}(t,x)$ to the function $\alpha(t,x)$, which plays an essential role in the algorithm \eqref{eq:algo}.
\begin{lemma}\label{lem:ineg} There exists a constant $\kappa>0$ (depending only on the space dimension $d$) such that
\begin{equation}\label{eq:lem:ineg}
\mathbb{U}(t,x)\le \kappa -\frac{1}{2}\log(\alpha(t,x)),\quad \forall (t,x)\in \mathbb{R}_+\times \mathcal{D}.
\end{equation}
\end{lemma}
\begin{proof} On one hand, the definition of $\alpha(t,x)$ in \eqref{eq:def:alp} implies that $\frac{e}{2d}\, \delta^2(x,\partial \mathcal{D})\ge \alpha(t,x)$ and consequently
\[
\log\frac{3}{\delta(x,\partial \mathcal{D})}\le \frac{1}{2}\log \frac{9e}{2d}-\frac{1}{2}\log \alpha(t,x).
\]
Using the property \eqref{eq:riesz-sup}, we obtain
\[
U(x)\le -\frac{1}{2}\log \alpha(t,x)+\frac{1}{2}\log \frac{9e}{2d}+\frac{1}{d}.
\]
On the other hand, the definition of $\alpha(t,x)$ also implies
\[
\frac{1}{2}\log\frac{3}{t}\le \frac{1}{2}\log\frac{3}{\alpha(t,x)}=\frac{\log 3}{2}-\frac{1}{2}\log\alpha(t,x).
\]
Combining both inequalities, we deduce that $\mathbb{U}(t,x)=\max\{ \frac{1}{2}\log(3/t), U(x) \}$ satisfies \eqref{eq:lem:ineg} with 
 $\kappa:=\max\{ \frac{1}{2}\log 3,\ \frac{1}{2}\log\frac{9e}{2d}+\frac{1}{d} \}$.
\end{proof}
An immediate consequence of Lemma \ref{lem:ineg} is an $L^2$-bound of $\mathbb{U}_n$, $n$ fixed. 
\begin{proposition}
\label{prop:L2bound}
Let $(T_0,X_0)=(t,x)$. There exist two constants $C_1$ and $C_2$ such that, 
\[
\mathbb{E}[\mathbb{U}_n^2]\le (C_1+C_2 n)^2,\quad\mbox{for}\quad n\ge \kappa-\frac{1}{2}\log(\alpha(t,x)).
\]
Here $\kappa$ stands for the constant defined in Lemma \ref{lem:ineg}.
\end{proposition}
\begin{proof}
Let us first recall that $\mathbb{U}_n:=\mathbb{U}(T_n,X_n)$. Due to the definition of the function $U$, we observe that $U(x)\ge 0$ for any $x\in\mathbb{R}^d$ and consequently $\mathbb{U}(t,x)\ge 0$ and $\mathbb{U}_n\ge 0$. Due to Lemma \ref{lem:ineg}, we shall focus our attention on $\log\alpha(T_n,X_n)$.\\
First we notice that \eqref{eq:algo} leads to
\begin{align*}
T_n=T_{n-1}-\alpha(T_{n-1},X_{n-1}) R_n\ge T_{n-1}(1-R_n).
\end{align*}
Hence
\begin{equation}
\label{eq:part1}
-\log(T_n)\le -\log(T_{n-1})-\log(1-R_n)\le -\log(\alpha(T_{n-1},X_{n-1}))-\log(1-R_n).
\end{equation}
Moreover by \eqref{eq:algo},
\[
\delta(X_{n},\partial \mathcal{D})\ge \delta(X_{n-1},\partial \mathcal{D})-2\sqrt{\alpha(T_{n-1},X_{n-1})}\psi(R_n).
\]
By its definition, $\alpha(t,x)\le \frac{e}{2d}\,\delta^2(x,\partial \mathcal{D})$, and we obtain
\[
\delta(X_{n},\partial \mathcal{D})\ge \delta(X_{n-1},\partial \mathcal{D})\Big( 1-\sqrt{\frac{2e}{d}}\ \psi(R_n) \Big),
\]
and therefore
\begin{equation}
\label{eq:part2}
-\log\Big( \frac{e}{2d}\ \delta^2(X_n,\partial \mathcal{D}) \Big)\le -\log(\alpha(T_{n-1},X_{n-1}))-2\log\Big(1-\sqrt{\frac{2e}{d}}\psi(R_n)\Big).
\end{equation}
Let us define $W_n:=-2\log\Big(1-\sqrt{\frac{2e}{d}}\psi(R_n)\Big)-\log(1-R_n)$. Combining \eqref{eq:part1} and \eqref{eq:part2} we finally obtain
\begin{equation}
\label{eq:fin}
-\log(\alpha(T_n,X_n))\le -\log(\alpha(T_{n-1},X_{n-1}))+W_n\le -\log(\alpha(t,x))+\sum_{k=1}^n W_k.
\end{equation}
Let us just note that $(W_n)$ is a family of independent and identically distributed random variables and $(t,x)$ is the starting position of the algorithm. Let us recall that $\mathbb{U}_n\ge 0$. We obtain
\begin{align*}
\mathbb{E}[\mathbb{U}_n^2]&\le\mathbb{E}\Big[ \Big(\kappa-\frac{1}{2}\log(\alpha(T_n,X_n))\Big)^2 \Big]\le \mathbb{E}\Big[ \Big( \kappa-\frac{1}{2}\log(\alpha(T_0,X_0)) +\sum_{k=1}^n W_k\Big)^2 \Big]\\
&\le 2\Big( \kappa-\frac{1}{2}\log(\alpha(t,x))\Big)^2 +2\mathbb{E}\Big[\Big(\sum_{k=1}^n W_k\Big)^2 \Big]\\
&\le  2\Big( \kappa-\frac{1}{2}\log(\alpha(t,x))\Big)^2 +2{\rm Var}\Big( \sum_{k=1}^n W_k \Big)+2\mathbb{E}\Big[\sum_{k=1}^n W_k \Big]^2\\
&\le 2\Big( \kappa-\frac{1}{2}\log(\alpha(t,x))\Big)^2  +2n {\rm Var}(W_1)+2n^2(\mathbb{E}[W_1])^2\\
&\le 2n {\rm Var}(W_1)+2n^2(\mathbb{E}[W_1]^2+1),
\end{align*}
due to the hypothesis $n\ge  \kappa-\frac{1}{2}\log(\alpha(t,x))$.
So Lemma \ref{lem:born} implies the statement of the Proposition \ref{prop:L2bound}: the upper-bound is quadratic with respect to $n$.
\end{proof}
Let us now point out a lower-bound for the expected value of the submartingale: $(\mathbb{E}[\mathbb{U}_n])_{n\ge 0}$.

\begin{proposition}\label{prop:lower:mean} There exist two constants $C_3\in\mathbb{R}$ and $C_4>0$, such that 
\begin{equation}
\label{eq:prop:lower:mean}
\mathbb{E}[\mathbb{U}_n]\ge C_3+C_4n,\quad n\ge 1.
\end{equation}
\end{proposition}
\begin{proof} Since $(\mathbb{U}_n)_{n\ge 0}$ is a submartingale we know that $(\mathbb{E}[\mathbb{U}_n])_{n\ge 0}$ is a non-decreasing sequence, but we need even more. In fact, due to the following lower-bound:
\[
\mathbb{U}_n\ge \frac{1}{4}\,\log\frac{3}{T_n}+\frac{1}{2}\, U(X_n)=:\mathbb{V}_n,
\] 
it suffices to point out the existence of a constant $L_0>0$ such that 
\begin{equation}\label{eq:introdeV}
\mathbb{E}[\mathbb{V}_{n+1}]-\mathbb{E}[\mathbb{V}_n]\ge L_0,\quad \forall
n\ge 0.
\end{equation}
In order to compute such a lower-bound, we consider two cases: either $\alpha(T_n,X_n)=T_n$ (event denoted by $\mathcal{T}_n$)
or $\alpha(T_n,X_n)\neq T_n$ (event denoted by $\overline{\mathcal{T}}_n$).\\[5pt]
\emph{Step 1.} First case: $\alpha(T_n,X_n)=T_n$. Then the definition of the random walk \eqref{eq:algo} implies that 
\[
T_{n+1}=T_n-\alpha(T_n,X_n)R_{n+1}=T_n(1- R_{n+1})\quad \mbox{on}\quad \mathcal{T}_n.
\]
Hence
\[
\log\frac{3}{T_{n+1}}=\log\frac{3}{T_{n}}-\log(1-R_{n+1})\quad \mbox{on}\quad \mathcal{T}_n.
\]
Let us denote by $L_1=-\frac{1}{4}\,\mathbb{E}[\log(1- R_1)]>0$. Since $U(X_n)$ is a submartingale ($U$ being subharmonic in $\mathcal{D}$), we get
\begin{align}
\label{eq:first}
\mathbb{E}\Big[ \mathbb{V}_{n+1}1_{\mathcal{T}_n} \Big]&=\mathbb{E}\Big[ \frac{1}{4}\,\log\frac{3}{T_{n+1}}1_{\mathcal{T}_n} \Big]+\frac{1}{2}\,\mathbb{E}\Big[\mathbb{E}[ U(X_{n+1})|\mathcal{F}_n]1_{\mathcal{T}_n} \Big]\nonumber \\
&\ge \mathbb{E}\Big[ \frac{1}{4}\,\log\frac{3}{T_{n}}1_{\mathcal{T}_n} \Big]+L_1\mathbb{P}(\mathcal{T}_n)+\frac{1}{2}\,\mathbb{E}\Big[U(X_{n})1_{\mathcal{T}_n} \Big]\nonumber \\
&\ge \mathbb{E}\Big[ \mathbb{V}_{n}1_{\mathcal{T}_n} \Big]+L_1\mathbb{P}(\mathcal{T}_n).
\end{align}
\emph{Step 2.} Second case: $\alpha(T_n,X_n)\neq T_n$. Let us recall that the random walk satisfies:
$X_{n+1}=X_n+2\sqrt{\alpha(T_n,X_n)}\psi_d(R_{n+1})V_{n+1}$ where $R_{n+1}$ is a continuous random variable whose support is the whole interval $[0,1]$ and whose distribution does not depend on $n$.  Observe also that 
\[
\rho(R_{n+1}):=\sqrt{\frac{2 e}{d}}\ \psi_d(R_{n+1})
\]
is also a continuous random variable with support $[0,1]$. In other words, on the event  $\alpha(T_n,X_n)\neq T_n$ and given $R_{n+1}=r$, the $(n+1)$-th step of the random walk is exactly the same as the $(n+1)$-th step of the classical random walk on spheres (see the Appendix \ref{sec:appen}) with radius $\beta=\rho(r)$, for which we can obtain some lower-bound. So using Proposition \ref{prop:classic}, we obtain
\begin{align}
\label{eq:first-bis}
\mathbb{E}\Big[ \mathbb{V}_{n+1}1_{\overline{\mathcal{T}}_n} \Big]&=\mathbb{E}\Big[ \frac{1}{4}\,\log\frac{3}{T_{n+1}}1_{\overline{\mathcal{T}}_n} \Big]+\frac{1}{2}\,\mathbb{E}[ U(X_{n+1})1_{\overline{\mathcal{T}}_n} ]\nonumber   \\
&\ge \mathbb{E}\Big[ \frac{1}{4}\,\log\frac{3}{T_{n}}1_{\overline{\mathcal{T}}_n} \Big] +\frac{1}{2}\,\mathbb{E}\Big[ \mathbb{E}[ U(X_{n+1})|\sigma(R_{n+1},T_n,X_n)]1_{\overline{\mathcal{T}}_n} \Big]\nonumber   \\
&\ge  \mathbb{E}\Big[ \frac{1}{4}\,\log\frac{3}{T_{n}}1_{\overline{\mathcal{T}}_n} \Big] +\frac{1}{2} \,\mathbb{E}[ (U(X_{n})+L1_{\{\delta/4<1- \rho(R_{n+1})<\delta/2 \}})1_{\overline{\mathcal{T}}_n} ]\nonumber   \\
&\ge \mathbb{E}\Big[ \mathbb{V}_{n}1_{\overline{\mathcal{T}}_n} \Big]+\frac{L}{2}\mathbb{P}(\{\delta/4<1- \rho(R_{n+1})<\delta/2\} \cap \overline{\mathcal{T}}_n)\\
&= \mathbb{E}\Big[ \mathbb{V}_{n}1_{\overline{\mathcal{T}}_n} \Big]+L_2 \mathbb{P}(\overline{\mathcal{T}}_n),
\end{align}
where $L_2=\frac{L}{2}\mathbb{P}(\delta/4<1- \rho(R_{1})<\delta/2)$, $R_{n+1}$ and $(T_n,X_n)$ being independent. Finally taking the sum of \eqref{eq:first} and \eqref{eq:first-bis}, we obtain \eqref{eq:introdeV} with $L_0=\min(L_1,L_2)>0$.
\end{proof}
We end here the preliminary results concerning the submartingale $(\mathbb{U}_n)_{n\ge 0}$. We are now ready to deal with the rate of convergence of the random walk on moving spheres.
\subsection{Rate of convergence of the algorithm}
Let us consider the algorithm $(T_n,X_n)_{n\ge 0}$ given by \eqref{eq:algo} and stopped as soon as $\alpha(T_n,X_n)\le\varepsilon$. We assume that the starting position satisfies $(T_0,X_0)=(t,x)\in\mathbb{R}_+\times\mathcal{D}$. Then the mean number of steps is bounded and the bound depends on $|\log\varepsilon|$.
\begin{proof}[Proof of Theorem \ref{thm:mean-number}]
 If the starting position $(t,x)$ satisfies $\alpha(t,x)\le \varepsilon$ then the algorithm stops immediately ($\mathcal{N}_\varepsilon=0$ a.s.) and the statement is satisfied. From now on, we assume that $\alpha(t,x)>\varepsilon$.\\[5pt]
\emph{Step 1. A remark on the stopping rule.} The statement of Theorem \ref{thm:mean-number} concerns $\mathcal{N}_\varepsilon$, see \eqref{def:neps}, the first time the random walk $(T_n,X_n)_{n\ge 0}$ hits a $\varepsilon$-neighborhood of the boundary. 
Let us introduce another stopping rule concerning $\mathbb{U}_n:=\mathbb{U}(T_n,X_n)$, $\mathbb{U}$ being defined by \eqref{eq:def:U}: 
\[
\mathcal{N}'_\varepsilon:=\inf\Big\{ n\ge 0:\ \mathbb{U}_n\ge\log\frac{3}{\varepsilon}+\frac{1}{d} \Big\}.
\]
Let us now point out that $\mathcal{N}_\varepsilon\le \mathcal{N}'_\varepsilon$ a.s. for $\varepsilon$ small enough (more precisely, we need $\varepsilon\le \frac{2d}{e}$). \\Indeed, let us consider the first case: $U(X_n)\ge \log\frac{3}{\varepsilon}+\frac{1}{d}$, then \eqref{eq:riesz-sup-2} implies that $\delta(X_n,\partial \mathcal{D})\le \varepsilon$. Moreover, due to the condition $\varepsilon\le \frac{2d}{e}$, we get $\frac{e}{2d}\ \delta^2(X_n,\partial \mathcal{D})\le \varepsilon$ and therefore $\alpha(T_n,X_n)\le \varepsilon$. \\
On the other side, if $\frac{1}{2}\log\frac{3}{T_n}\ge \log\frac{3}{\varepsilon}+\frac{1}{d}\ge \frac{1}{2}\log\frac{3}{\varepsilon}$ then $T_n\le \varepsilon$ and finally $\alpha(T_n,X_n)\le \varepsilon$. So we deduce that $\mathbb{U}_n\ge\log\frac{3}{\varepsilon}+\frac{1}{d}$ implies that $\alpha(T_n,X_n)\le \varepsilon$. In the sequel, we will find an upper-bound for the mean value of $\mathcal{N}'_\varepsilon$.\\ \\ \\
\emph{Step 2.} The aim of the second step is to prove the existence of an integer $\eta\in\mathbb{N}$ and a constant $p<1$ both independent with respect to the starting position of the random walk $(T_0,X_0)=(t,x)$ and independent of the parameter $\varepsilon$ such that 
\begin{equation}\label{eq:p}
\mathbb{P}(\mathcal{N}'_\varepsilon>\eta \lfloor-\log\varepsilon\rfloor)\le p,
\end{equation}
for $\varepsilon$ small enough. Let us note that the complementary event satisfies, by definition,
\begin{align*}
\mathbb{P}(\mathcal{N}'_\varepsilon\le \eta\lfloor -\log\varepsilon\rfloor)\ge \mathbb{P}(\mathbb{U}_{k}\ge \beta_k), 
\end{align*}
where $k=\eta\lfloor -\log\varepsilon\rfloor$ and $\beta_k=\log 3+1+\frac{1}{d}+k/\eta$. We deduce that there exists a particular choice of the integer $\eta$ such that, for $\varepsilon$ small enough, $\beta_k< \alpha_k:=(C_3+C_4k)/2$ where $C_3$ and $C_4$ are defined in Proposition \ref{prop:lower:mean}.  So it is sufficient to find a lower-bound of $\mathbb{P}(\mathbb{U}_{k}> \alpha_k)$ which should be positive when $k$ is large. By Proposition \ref{prop:L2bound}, there exist two constants $C_1$ and $C_2$ such that 
\[
\mathbb{E}[\mathbb{U}_n^2]\le (C_1+C_2 n )^2, \quad \mbox{for any}\ n\ge \kappa-\frac{1}{2}\log(\alpha(t,x)). 
\]
Due to the condition on the initial position $\alpha(t,x)>\varepsilon$, the previous inequality is satisfied for $n\ge \lfloor -\log \varepsilon\rfloor$ when $\varepsilon$ is small enough. In particular, it is satisfied for $n=k= \eta\lfloor -\log \varepsilon\rfloor$.
We obtain
\begin{align}\label{eq:decomp1}
\mathbb{E}[\mathbb{U}_k]&=\mathbb{E}\Big[\mathbb{U}_k1_{\{ \mathbb{U}_k\le\alpha_k \}}\Big]+\mathbb{E}\Big[\mathbb{U}_k1_{\{ \mathbb{U}_k>\alpha_k\}}\Big]>C_3+C_4k.
\end{align}
Then by an application of \eqref{eq:prop:lower:mean} and the Cauchy-Schwarz inequality, we get
\[
\alpha_k+\sqrt{\mathbb{E}[\mathbb{U}_k^2]}\sqrt{\mathbb{P}(\mathbb{U}_k>\alpha_k)}>C_3+C_4 k.
\]
Therefore, due to the upper-bound of the second moment,
\[
\alpha_k+(C_1+C_2k)\sqrt{\mathbb{P}(\mathbb{U}_k>\alpha_k)}>C_3+C_4 k.
\]
We deduce 
\[
\mathbb{P}(\mathbb{U}_k>\alpha_k)\ge \frac{1}{4}\left(\frac{C_3+C_4k}{C_1+C_2k}\right)^2>\frac{1}{5}\left(\frac{C_4}{C_2}\right)^2,
\]
for $k$ large enough that is $\varepsilon$ small enough. This implies the existence of the constant $p>0$ in \eqref{eq:p}.\\[5pt]
\emph{Step 3. Upper-bound of} $\mathbb{E}[\mathcal{N}'_\varepsilon]$. Due to the first step it is sufficient to obtain an upper-bound of $\mathbb{E}[\mathcal{N}'_\varepsilon]$ in order to prove the statement of the theorem. Such a result is essentially based on the Markov property of the sequence $(T_n,X_n)_{n\ge 0}$: the second step implies in particular that
\[
\mathbb{P}(\mathcal{N}'_\varepsilon>k\eta\lfloor -\log\varepsilon\rfloor)\le p^k,\quad \forall k\ge 1.
\]
Hence
\begin{eqnarray*}
\mathbb{E}[\mathcal{N}'_\varepsilon]&\le \ds\sum_{k\ge 1}&k\eta\lfloor -\log\varepsilon\rfloor \mathbb{P}\Big(\mathcal{N}'_\varepsilon\le k\eta \lfloor -\log\varepsilon\rfloor\Big|\mathcal{N}'_\varepsilon> (k-1)\eta \lfloor -\log\varepsilon\rfloor\Big)\\
&&\times\mathbb{P}(\mathcal{N}'_\varepsilon> (k-1)\eta\lfloor - \log\varepsilon\rfloor)\\
&&\le \eta \lfloor -\log\varepsilon\rfloor \sum_{k\ge 1}kp^{k-1}=\frac{\eta | \log\varepsilon|}{(1-p)^2}.
\end{eqnarray*}
\end{proof}
\section{Examples and numerics}
The aim of this section is to illustrate the random walk on spheres algorithm introduced in Section \ref{sec:approx}. Let us focus our attention on the numerical approximation of the solution to the value problem:
\begin{equation}
\left\{\begin{array}{ll}
\partial_t u(t,x)-\Delta_x u(t,x)=0, & \forall (t,x)\in \mathbb{R}_+\times\mathcal{D,}\\
u(t,x)=f(t,x),&\forall (t,x)\in\mathbb{R}_+ \times\partial\mathcal{D},\\
u(0,x)=f_0(x), & \forall\ x\in\mathcal{D},
\end{array}\right.
\end{equation}
for particular domains $\mathcal{D}$. First we shall present results obtained for the hypercube $\mathcal{D}=]0,1[^d$ and secondly the half of a sphere $\mathcal{D}=\{x\in\mathbb{R}^d:\ \Vert x\Vert< 1,\ x_1> 0\}$. Of course these toy examples are not directly related to concrete situations in physics but they permit to emphasize the efficiency of the algorithm. Their advantage relies in the easy computation of the distance to the boundary. For more general situations, only this part of the procedure has to be modified and can sometimes become quite painful.
\subsection{Hypercube}
Let us first introduce the functions which take part to the boundary conditions. We  choose  a function with the following simple expression 
\begin{equation}
\label{eq:hyp1}
f(t,x)=e^{t}\prod_{i=1}^d x_i(1-x_i)1_{]0,1[}(x_i),\quad x\in\overline{\mathcal{D}},\ \forall t\ge 0. 
\end{equation} 
Setting $f_0(x)=f(0,x)$, we observe that both the compatibility and the continuity conditions are obviously satisfied. In this particular case, we have already pointed out, in the previous sections, that there exists a unique (smooth) solution to the Initial-Boundary Value  Problem which can be approximated using the algorithm of moving spheres.\\

The solution can be approximated by  $u^\epsilon$ defined by \eqref{eq:approx}, the error being directly related to the parameter  $\epsilon$. Since $u^\epsilon(t,x)$ is the expectation of a random variable, we shall use a Monte-Carlo method in order to obtain an estimated value. Hence

\begin{equation}\label{eq:approx-Monte}
u^\varepsilon_N(t,x)=\frac{1}{N}\sum_{k=1}^N f(T_{\varepsilon,k},X_{\varepsilon,k})1_{\{ X_{\varepsilon,k}\in\partial \mathcal{D} \}}+\frac{1}{N} \sum_{k=1}^Nf_0(X_{\varepsilon,k})1_{\{  X_{\varepsilon,k}\notin\partial \mathcal{D} \}},
\end{equation}
where $(T_{\varepsilon,k},X_{\varepsilon,k})_{k\ge 0}$ is a sequence of independent and identically distributed couples of random variables, the distribution being defined at the begining of Section \ref{sec:approx}. The difference between $u(t,x)$ and $u^\varepsilon_N(t,x)$ actually relies on both the error described in Proposition \ref{thm:approx} of order $\sqrt{\varepsilon}$ on one hand and the classical error of Monte Carlo methods of order $N^{-1/2}$ on the other hand (the confidence interval depends as usual on the standard deviation of the underlying random variable).\\
First let us present $u^\varepsilon_N(t,x)$ for a particular point: the center of the hypercube ($x=(0.5,\ldots,0.5)$ is the default setting in all this subsection) letting the time cross the whole interval $[0,2]$. 
\begin{figure}[ht]
\centerline{\includegraphics[height=6cm]{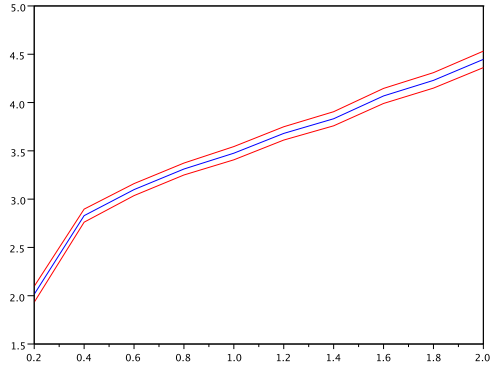}}
\caption {$u^\varepsilon_N(t,x)$ versus $t$ for $N=1\,000$, $\varepsilon=0.001$, $d=3$.}
 \label{fig0}
\end{figure}

We present at the same time the associated Monte-Carlo $95\%$-confidence interval (Figure \ref{fig0}). Let us just notice that the choice $N=1\,000$ is not motivated by some computational facilities but rather to produce a clear picture, the confidence interval becoming very small for larger values of $N$.   Of course the numerical method permits to observe directly the distribution of the random variable
\[
\mathcal{Z}_\varepsilon
=f(T_{\varepsilon},X_{\varepsilon})1_{\{ X_{\varepsilon}\in\partial \mathcal{D} \}}+f_0(X_{\varepsilon})1_{\{  X_{\varepsilon,k}\notin\partial \mathcal{D} \}},
\] 
which drastically changes as time elapses (Figure \ref{fig1} and \ref{fig2}).\\

\begin{figure}[ht]
\centerline{\includegraphics[height=6cm]{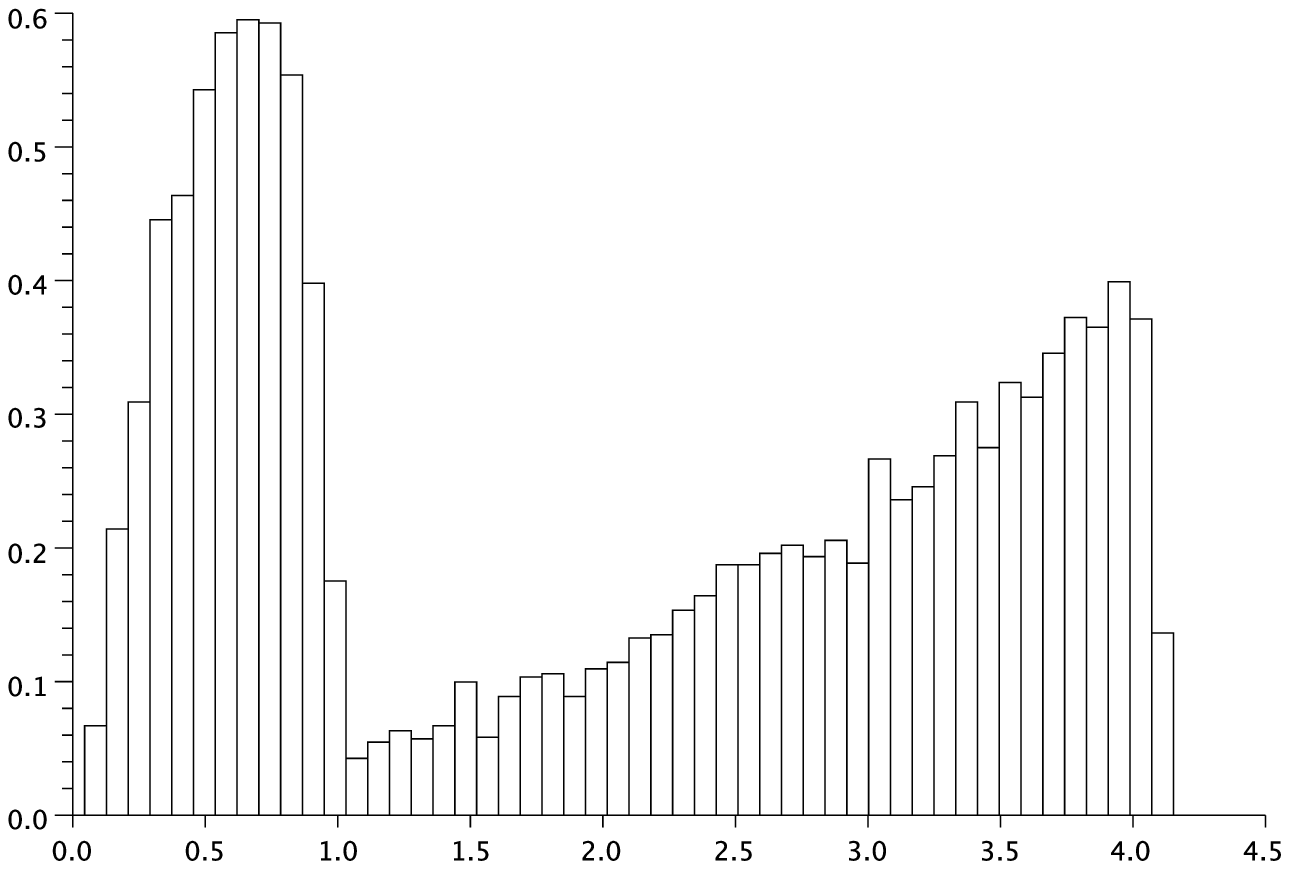}\hspace*{0.4cm}
\includegraphics[height=6cm]{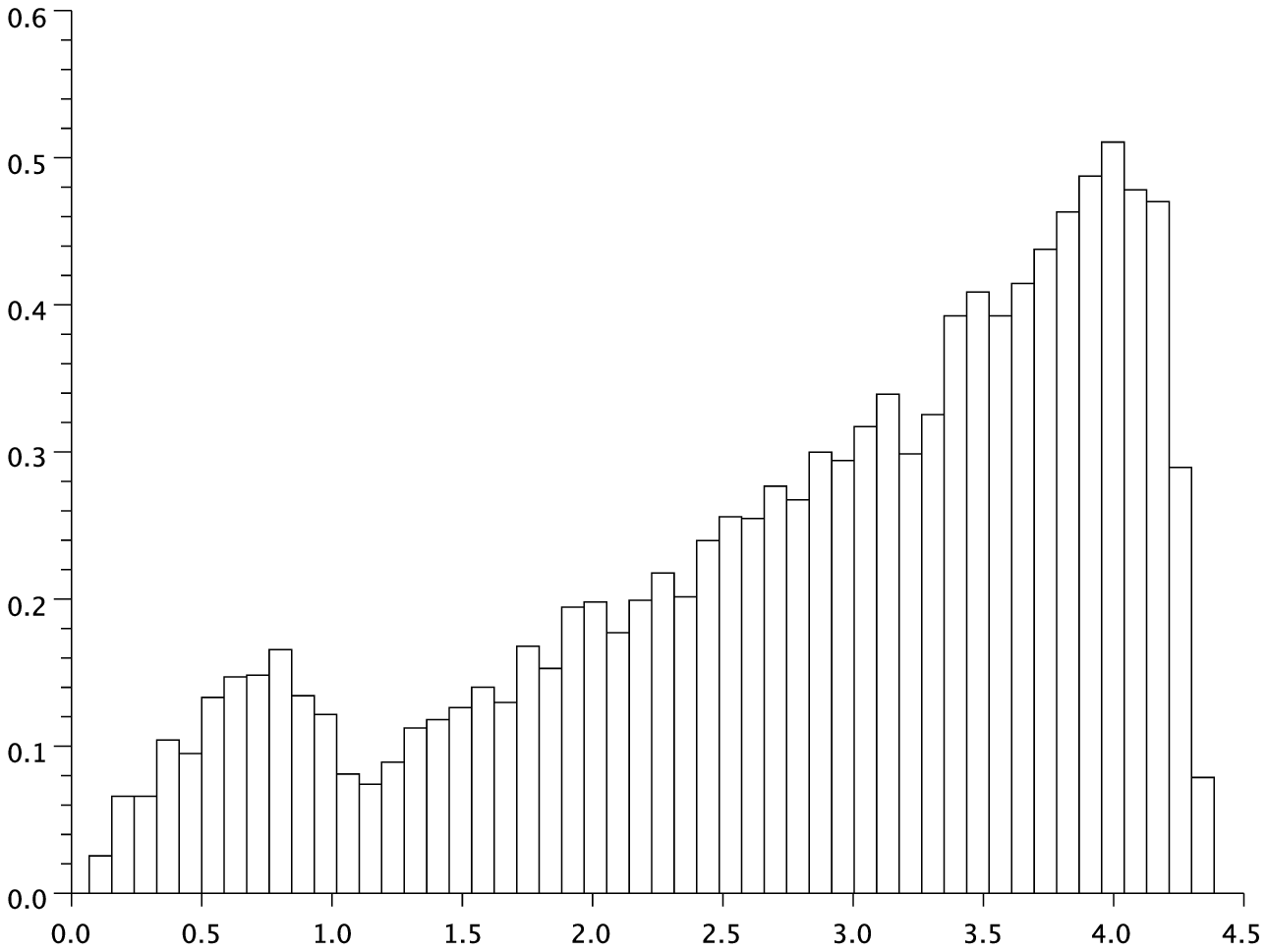}}
 \caption {Histogram of the distribution of $10\,000$ random variables $\mathcal{Z}_\varepsilon$ for various values of $t$: $t=0.05$ (left), $t=0.1$ (right) $\varepsilon=0.001$, $d=3$.  
 }
 \label{fig1}
\end{figure}

\begin{figure}[ht]
\centerline{\includegraphics[height=6cm]{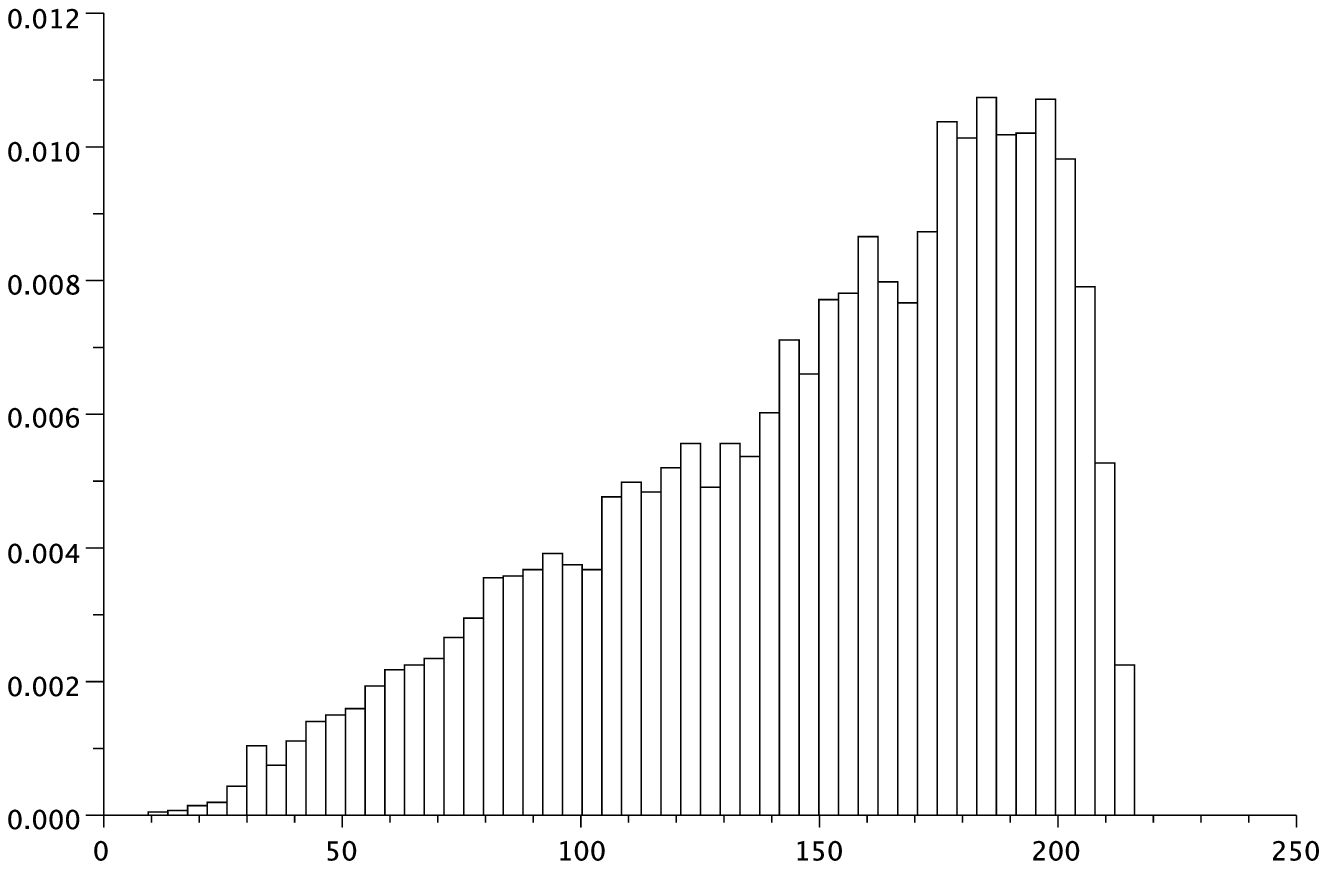}\includegraphics[height=6cm]{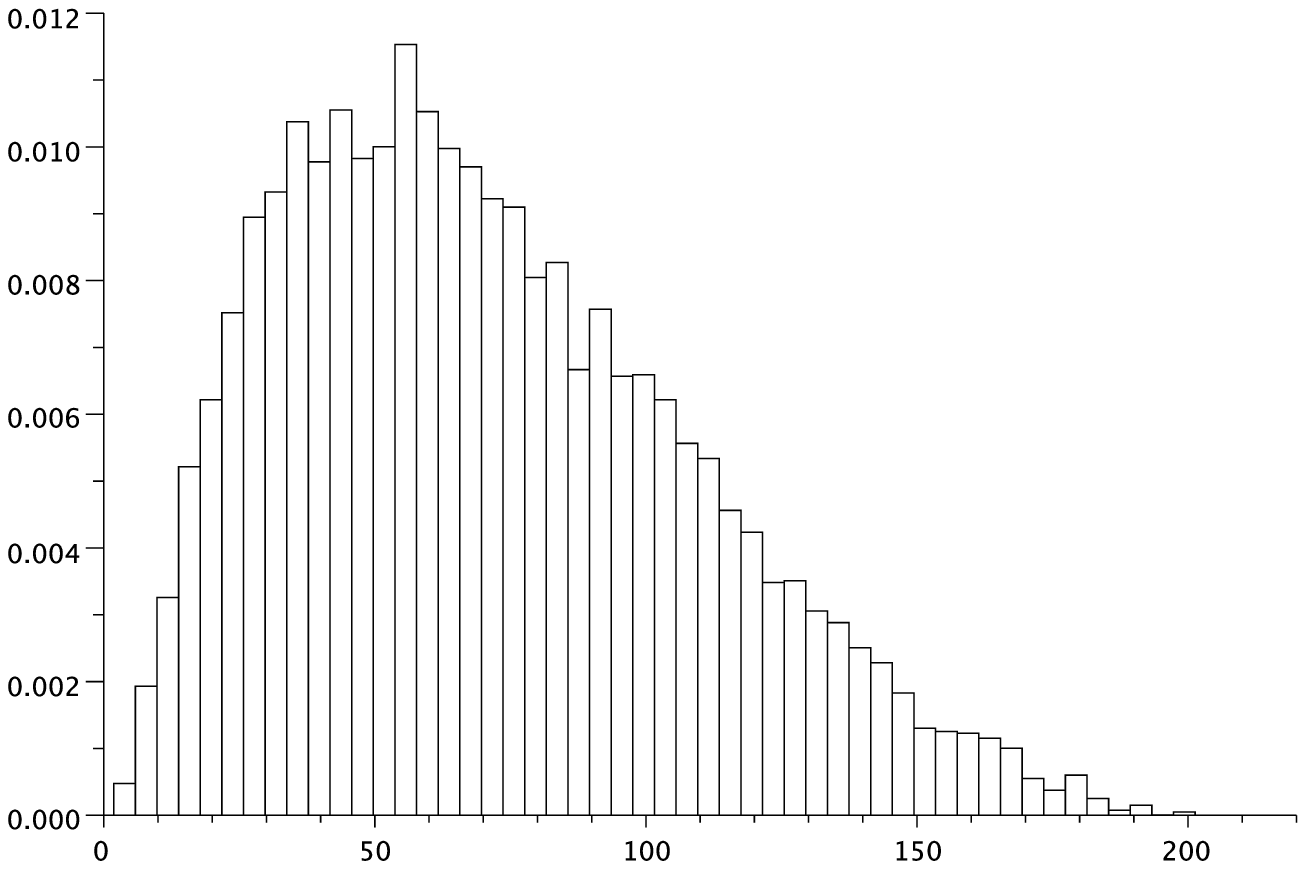}}
 \caption {Histogram of  $\mathcal{Z}_\varepsilon$ for $d=3$ (left) and $d=10$ (right),  $t=4$, $\varepsilon=0.001$.  
 }
 \label{fig2}
\end{figure}

 In our example, small values of $\mathcal{Z}_\varepsilon$ are more frequently observed for small time values than for large ones. Such behaviour of the random variable is not linked to the particular boundary conditions we introduced, but relies on the following general argument. The random variable $\mathcal{Z}_\varepsilon$ is obtained due to a stopping procedure on $M_n=(T_n,X_n)$ defined by \eqref{eq:algo}. The sequence is stopped as soon as either $X_n$ is $\varepsilon$-close to the boundary $\partial \mathcal{D}$ (we call this event \emph{stop due to space constraint}) or $T_n$ is $\varepsilon$-close to $0$ (\emph{stop due to time constraint}). Then it seems quite obvious that stops due to time constraint are more likely to occur when $t$ becomes small (see the proportion in Figure \ref{fig3} left). 

\begin{figure}[ht]
\centerline{\includegraphics[height=5.2cm]{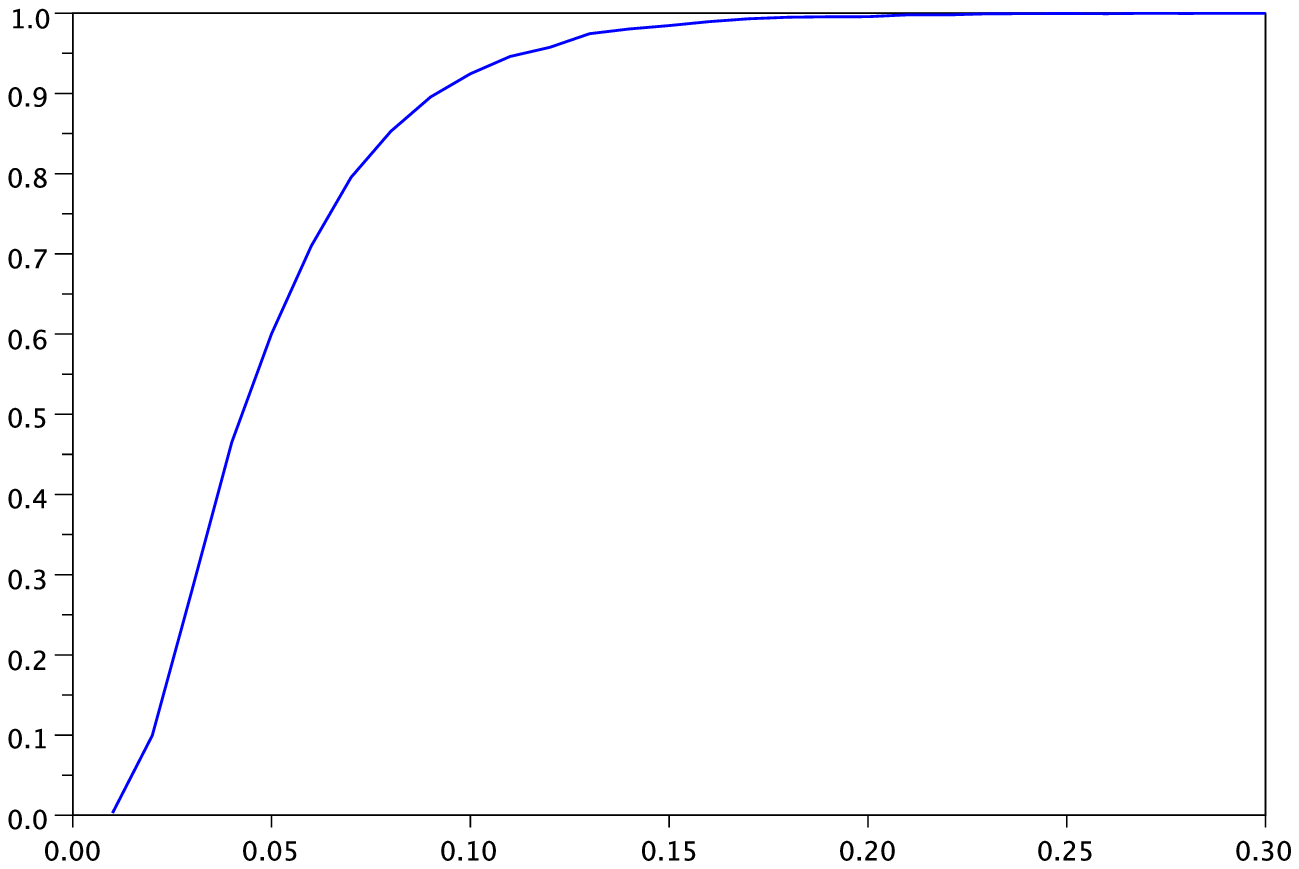}\hspace*{0.4cm}\includegraphics[height=5.2cm]{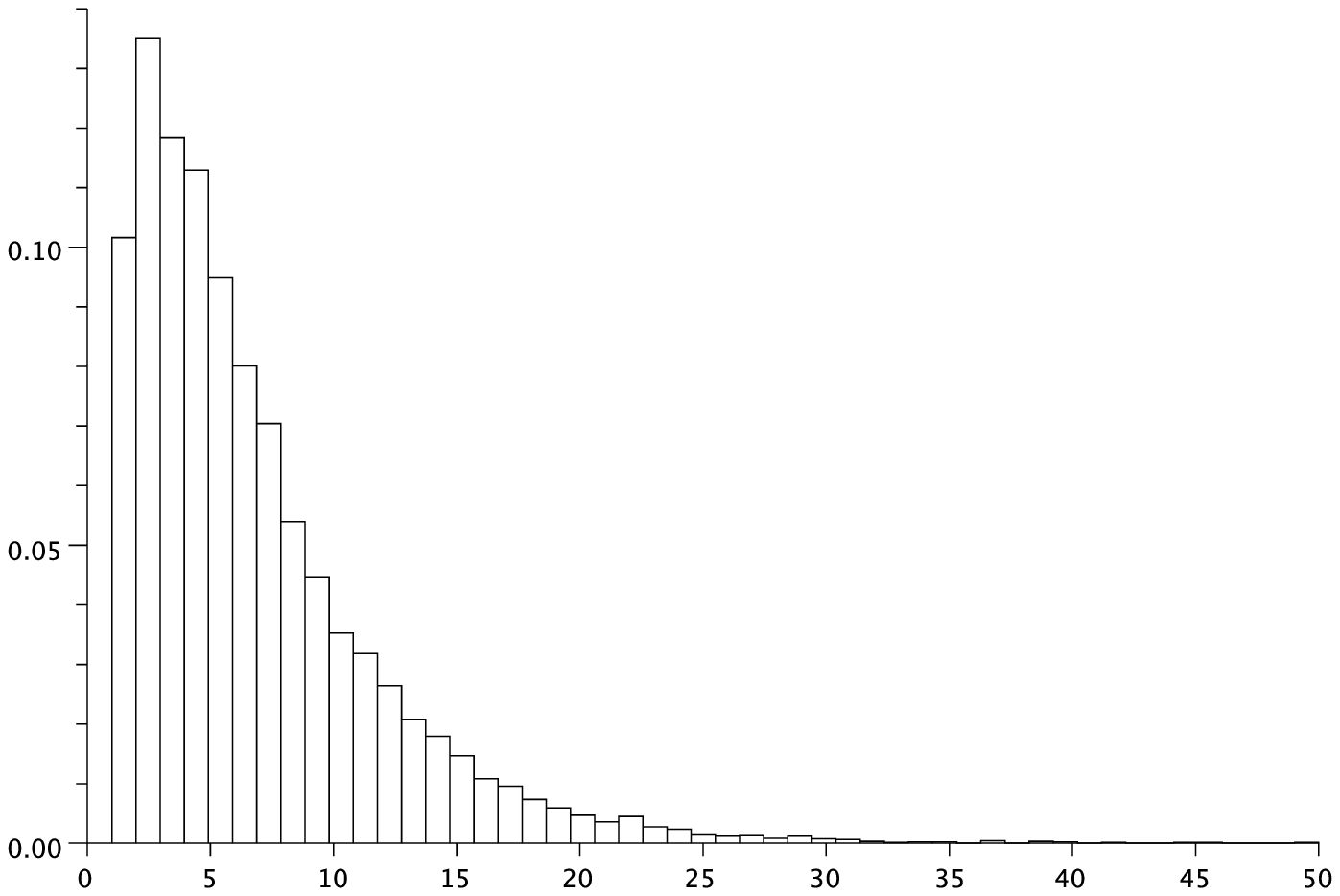}}
\caption {Proportion of stops due to space constraint versus $t$ for $10\,000$ trials, $\varepsilon=0.001$, $d=3$ (left); histogram of the number of steps $t=1$, $\varepsilon=0.001$, $d=3$.}
 \label{fig3}
\end{figure}
Let us now comment the algorithm efficiency by focusing our attention on the number of steps. The distribution of this random number depends on several parameters: the dimension $d$, the parameter $\varepsilon$ and finally the choice of $(t,x)$ (see histogram Figure \ref{fig3} -- right -- for a particular choice of parameters). We have pointed out an upper bound for the average number of steps in Theorem \ref{thm:mean-number}. The numerics permit to present different curves illustrating all the dependences: the logarithm growth with respect to the parameter $\varepsilon$, the surprising behavior when the space position $x$ varies and the influence of the dimension (Figures \ref{fig4} and \ref{fig5}). 
\begin{figure}[ht]
\centerline{\includegraphics[height=5.5cm]{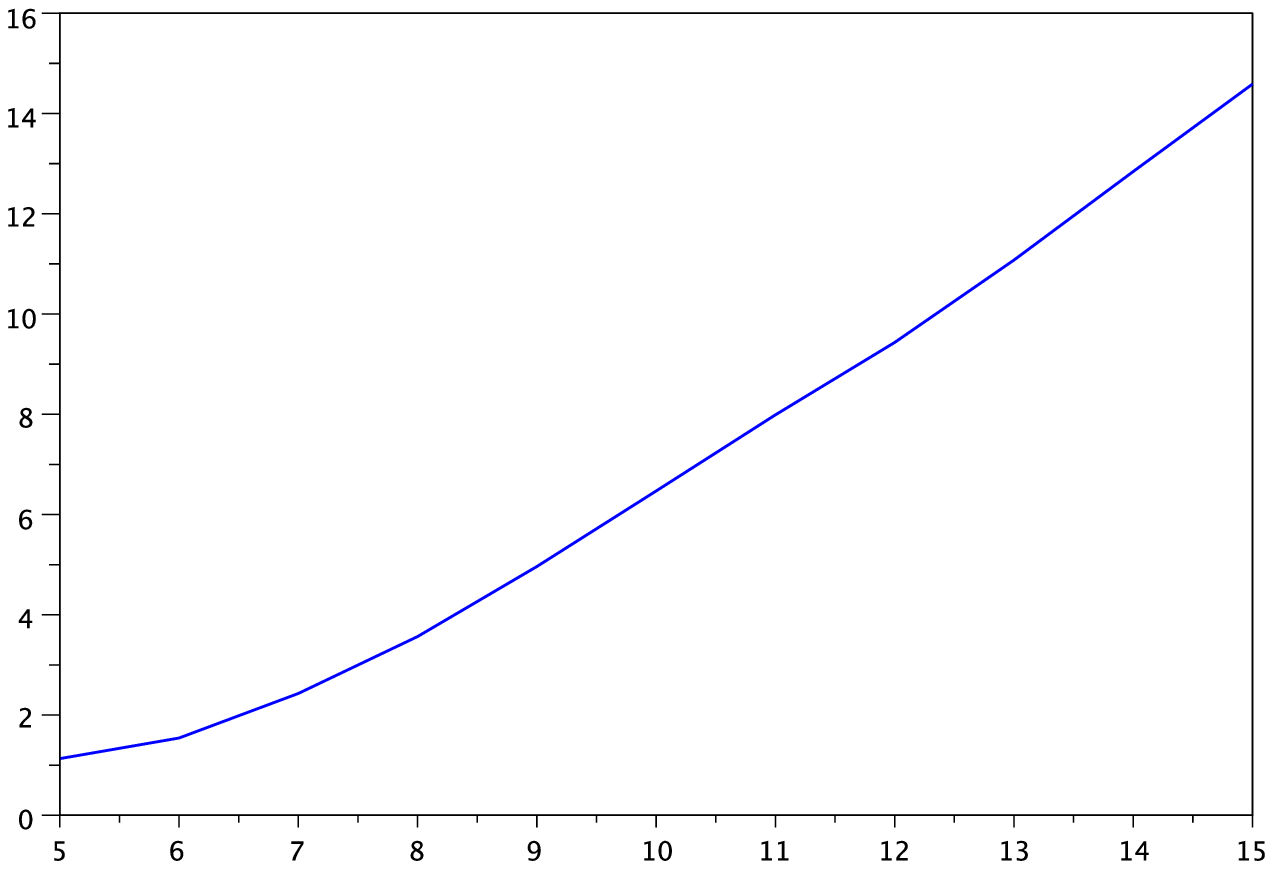}\hspace*{0.4cm}\includegraphics[height=5.5cm]{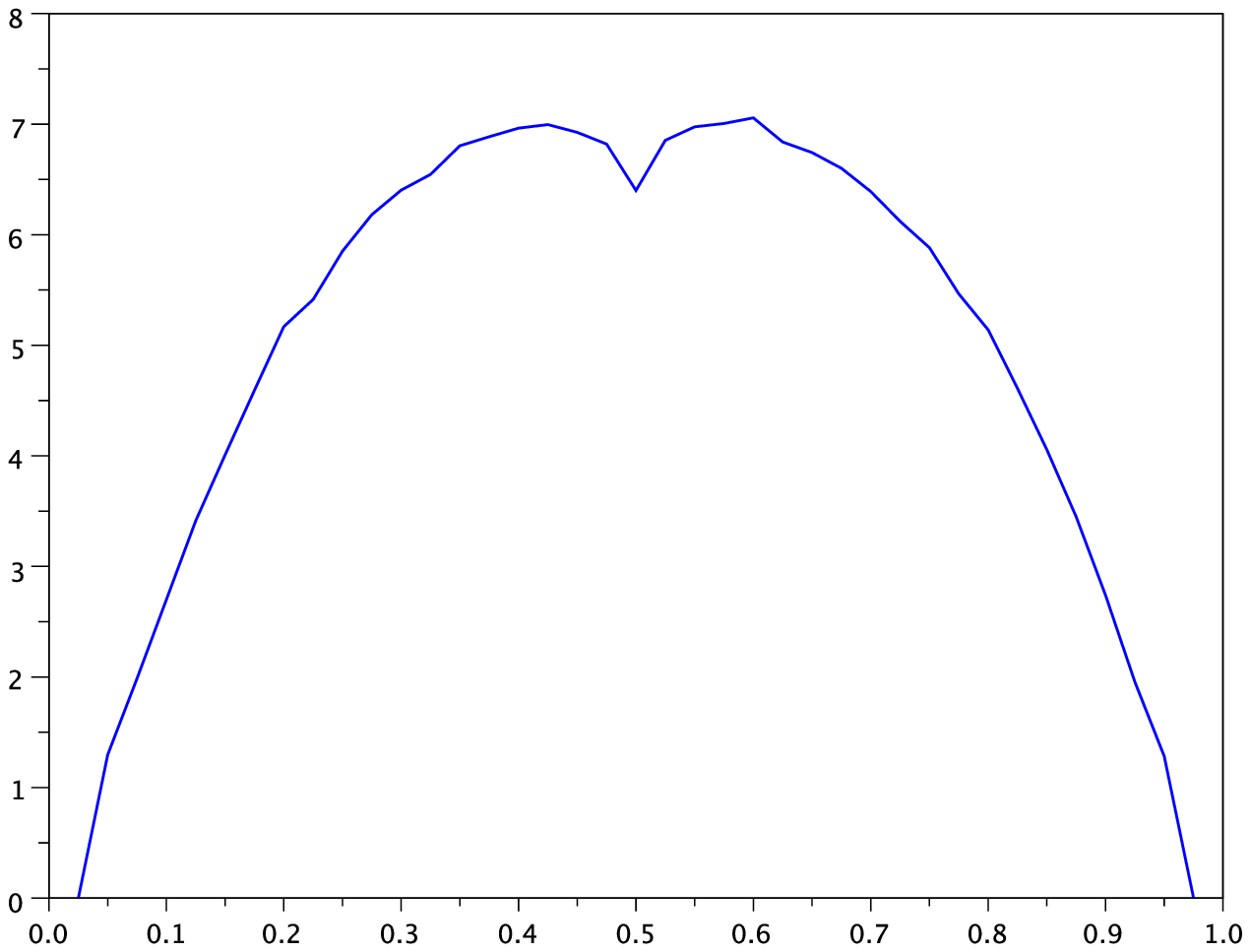}}
\caption {Averaged number of steps versus $n$ for $\varepsilon=0.5^n$ and $x$ the center of the hypercube (left), averaged number of steps versus $u$ for $x=(u,u,u)$ and $\varepsilon=0.001$. In both situations: $10\,000$ trials, $d=3$, $t=1$.}
 \label{fig4}
\end{figure}
\begin{figure}[ht]
\centerline{\includegraphics[height=5.5cm]{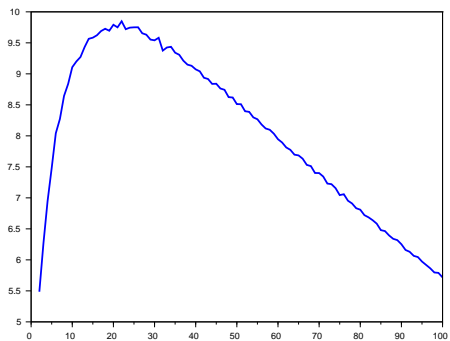}\hspace*{0.4cm}}
\caption {Averaged number of steps versus the dimension $d$, $\varepsilon=0.001$, $10\,000$ trials, $t=1$.}
 \label{fig5}
\end{figure}
Let us notice that this algorithm is especially efficient (see the small number of steps) even in high dimensions. 
\clearpage
\subsection{Half sphere}
All the studies developed in the hypercube case can also be considered for the half sphere. We introduce particular boundary conditions:
\begin{equation}
\label{eq:hyp2}
f(t,x)=(1+\cos(2\pi t))\Vert x\Vert,\quad \forall
x\in\overline{\mathcal{D}},\ \forall t\ge 0, 
\end{equation} 
with $f_0(x)=f(0,x)$. Similarly as above, we present:
\begin{itemize} 
\item the approximated solution $u^\varepsilon_N(t,x)$ for the default value  $x=(0.5,0\ldots,0)$ and for $t$ varying in the interval $[0,4]$ (Figure \ref{fig0.1}),
\item the distribution of the Monte Carlo underlying variable $\mathcal{Z}^\varepsilon$ for different values of $t$ and different dimension values $d$ (Figures \ref{fig1.1} and \ref{fig2.1}),
\item different curves illustrating the influence of the parameter $\varepsilon$, the starting position $x$ and the dimension $d$ on the averaged number of steps (Figures \ref{fig3.1} and \ref{fig4.1})
\end{itemize}
\begin{figure}[ht]
\centerline{\includegraphics[height=5cm]{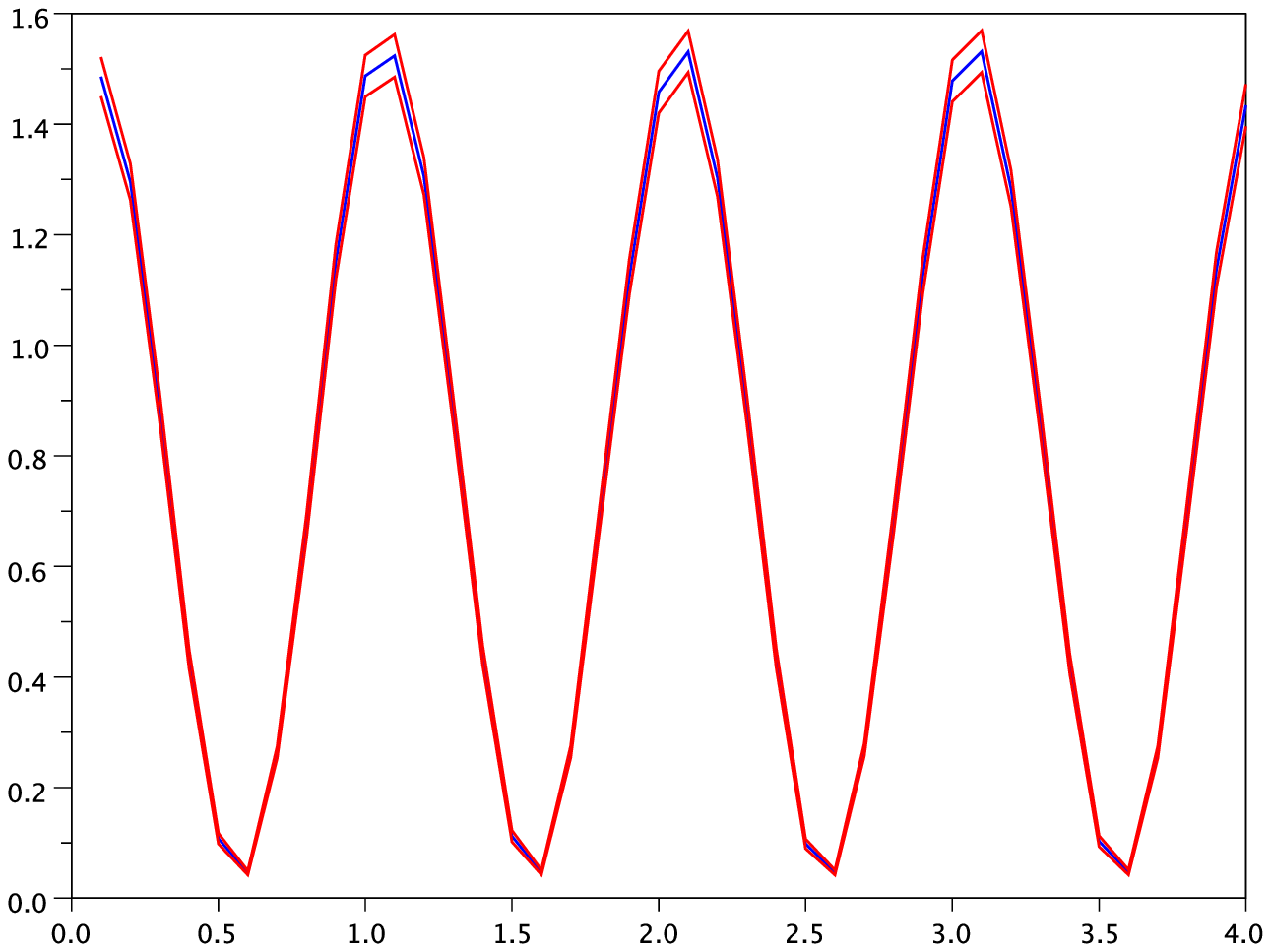}}
\caption {$u^\varepsilon_N(t,x)$ versus $t$ for $N=1\,000$, $\varepsilon=0.001$, $d=3$.}
 \label{fig0.1}
\end{figure}
\begin{figure}[ht]
\centerline{\includegraphics[height=4.5cm]{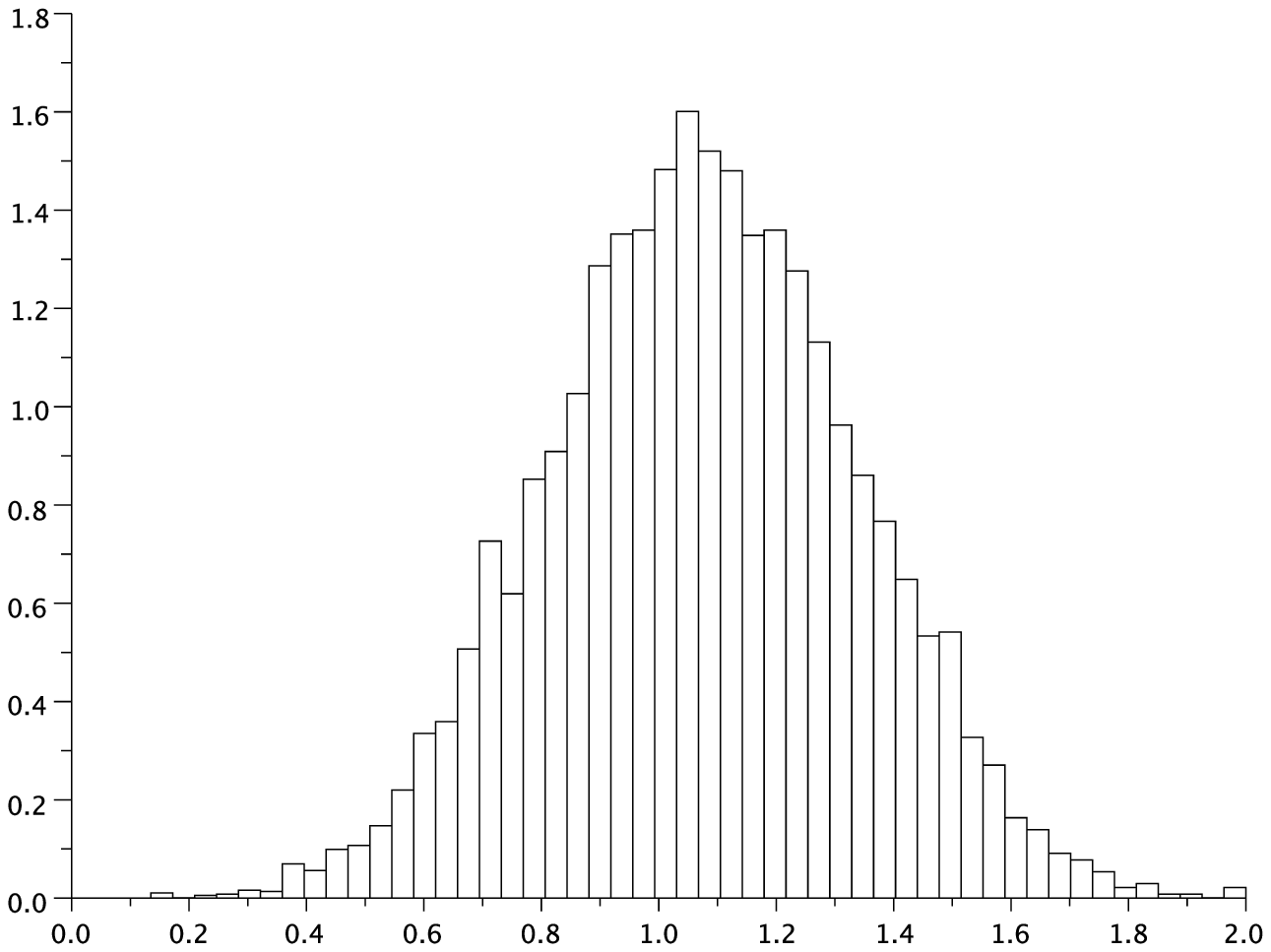}\hspace*{0.4cm}
\includegraphics[height=4.5cm]{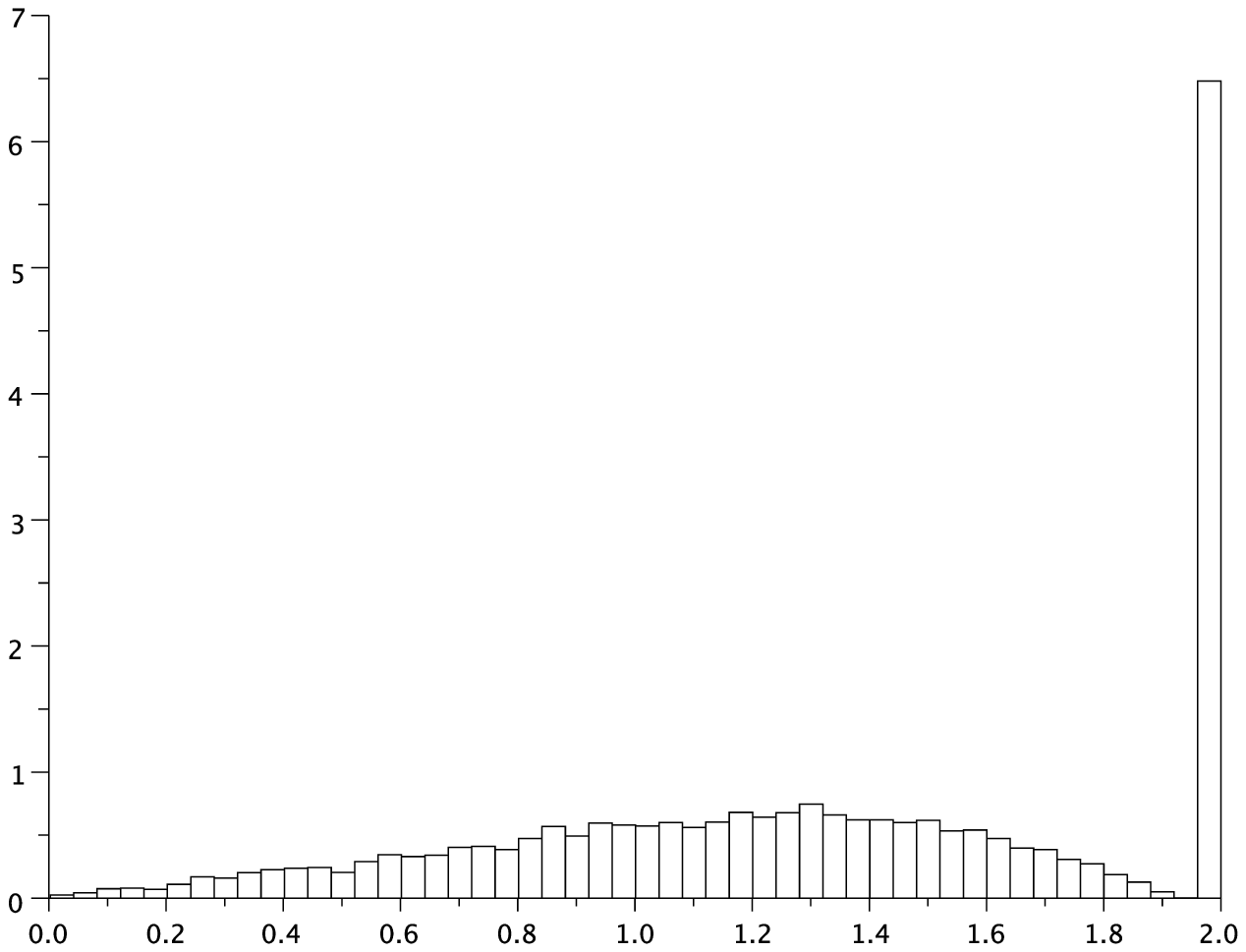}}
 \caption {Histogram of the distribution of $10\,000$ random variables $\mathcal{Z}_\varepsilon$ for various values of $t$: $t=0.01$ (left), $t=0.05$ (right) $\varepsilon=0.001$, $d=3$.  
 }
 \label{fig1.1}
\end{figure}
\begin{figure}[ht]
\centerline{\includegraphics[height=5cm]{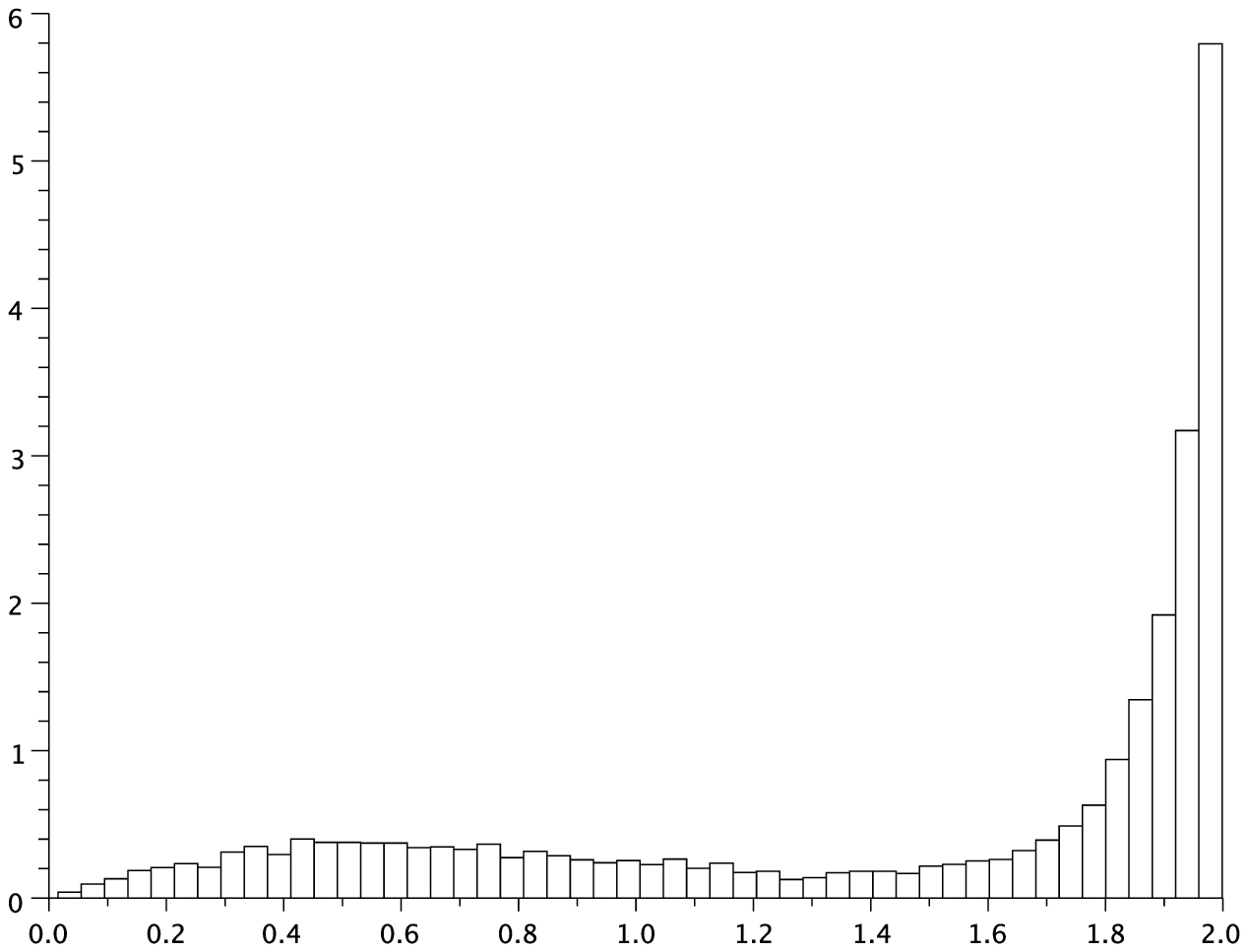}\hspace*{0.4cm}
\includegraphics[height=5cm]{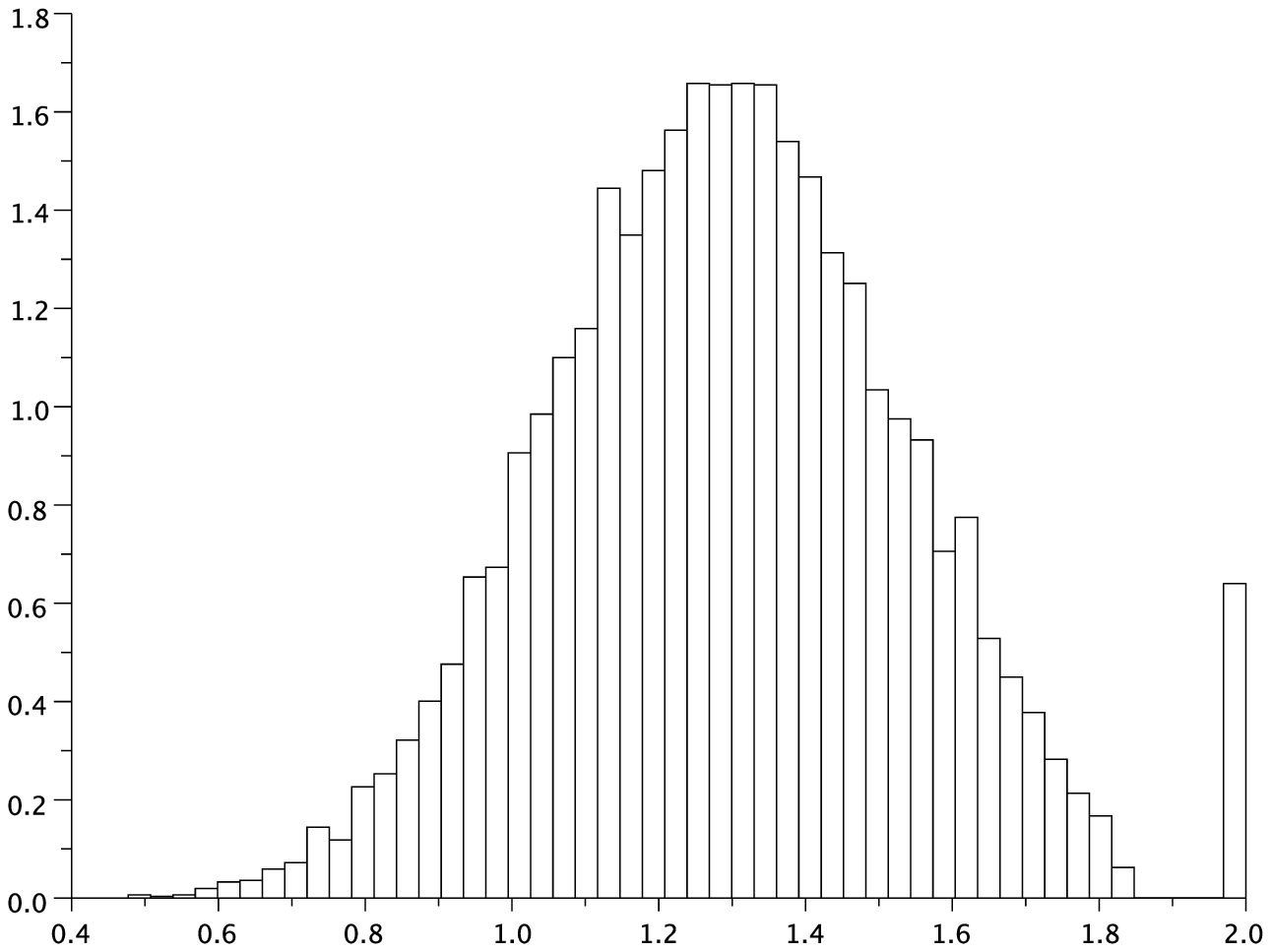}}
 \caption {Histogram of $\mathcal{Z}_\varepsilon$ for $d=3$ and $t=1$ (left), for $d=10$ and $t=0.01$ (right), both with $\varepsilon=0.001$.  
 }
 \label{fig2.1}
\end{figure}
\begin{figure}[ht]
\centerline{\includegraphics[height=5cm]{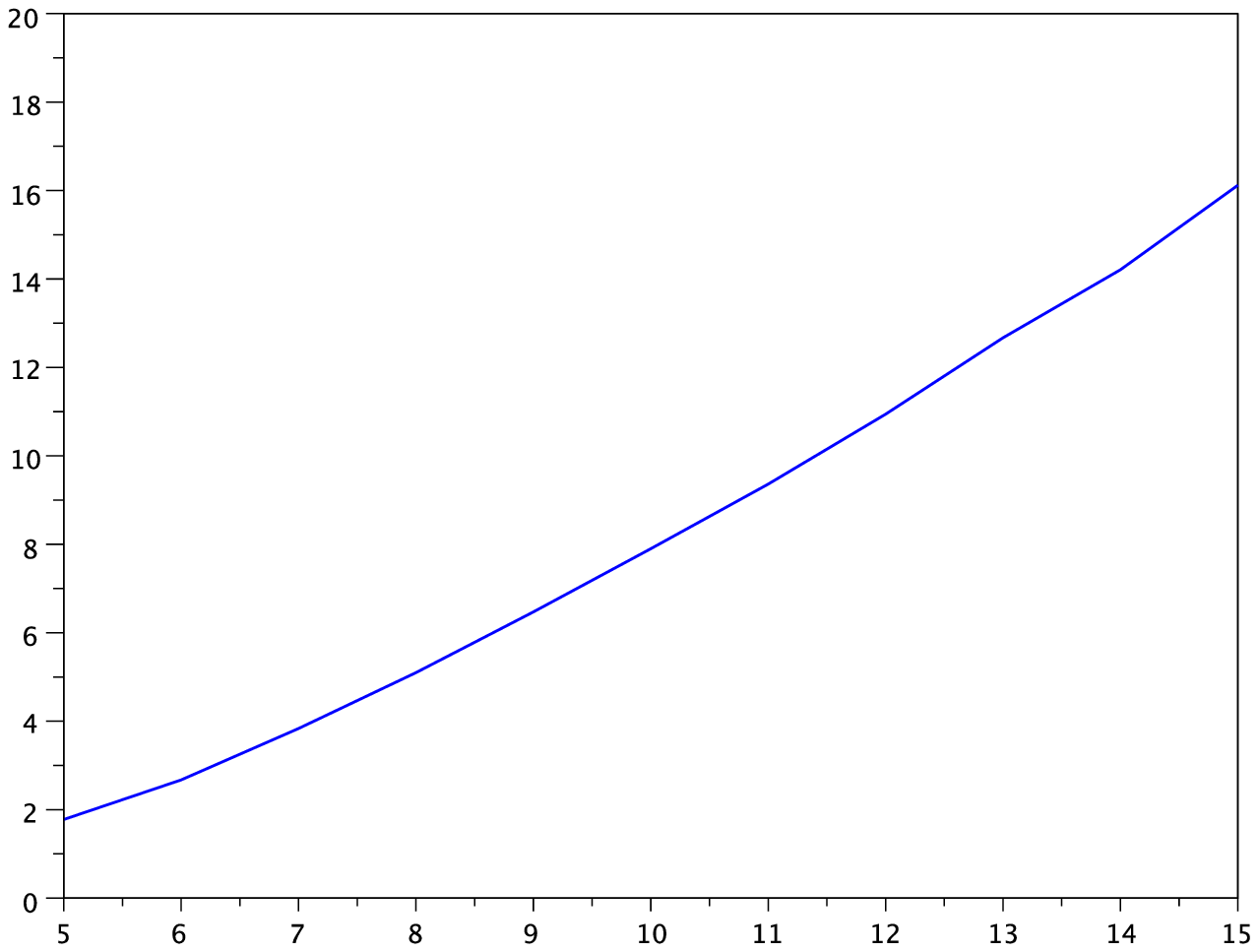}\hspace*{0.4cm}
\includegraphics[height=5cm]{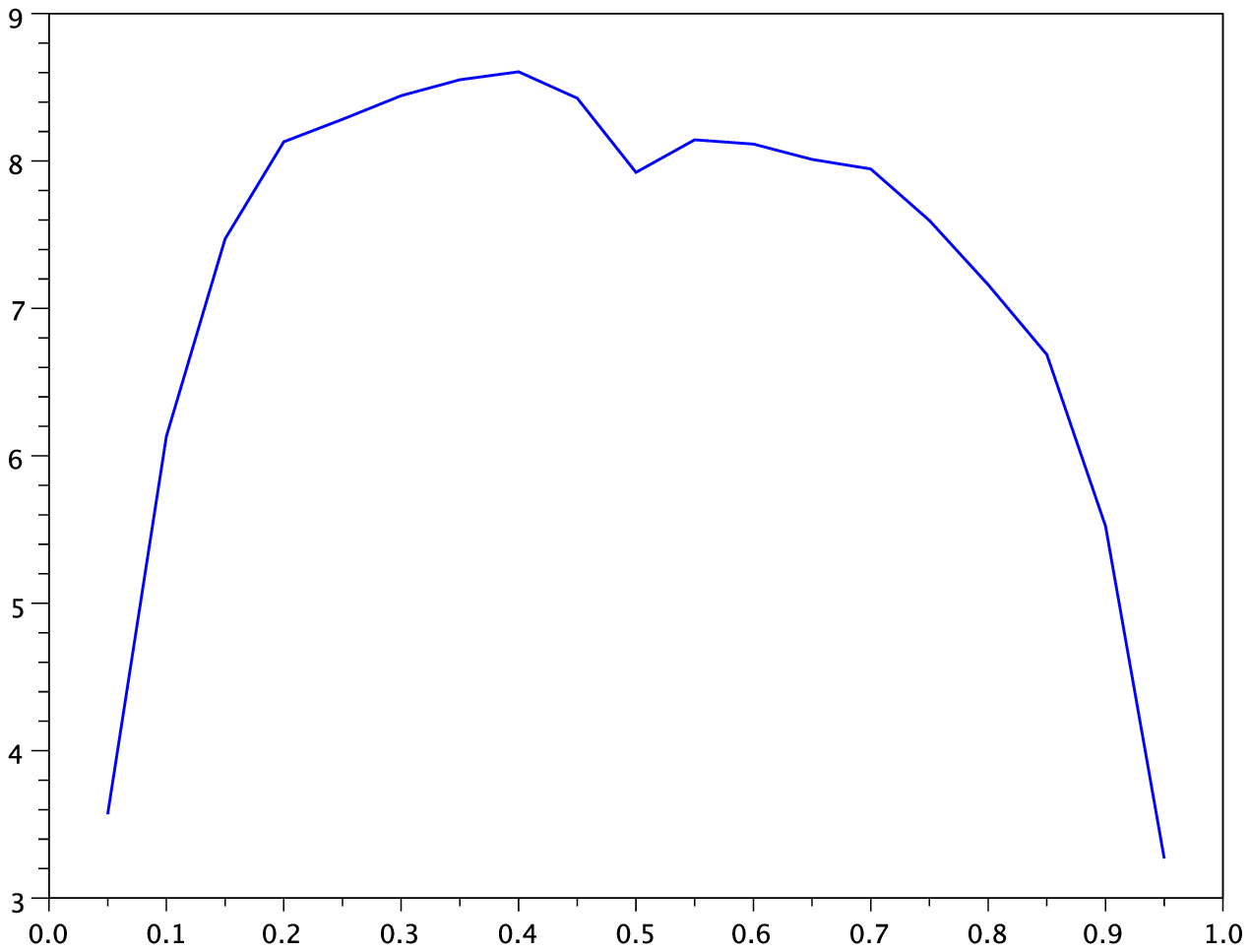}}
 \caption {Averaged number of steps versus $n$ for $\varepsilon=0.5^n$ and $x=(0.5,0,0)$  (left), averaged number of steps versus $u$ for $x=(u,0,0)$ and $\varepsilon=0.001$. In both situations: $10\,000$ trials, $d=3$, $t=1$.
 }
 \label{fig3.1}
\end{figure}
\begin{figure}[ht]
\centerline{\includegraphics[height=5cm]{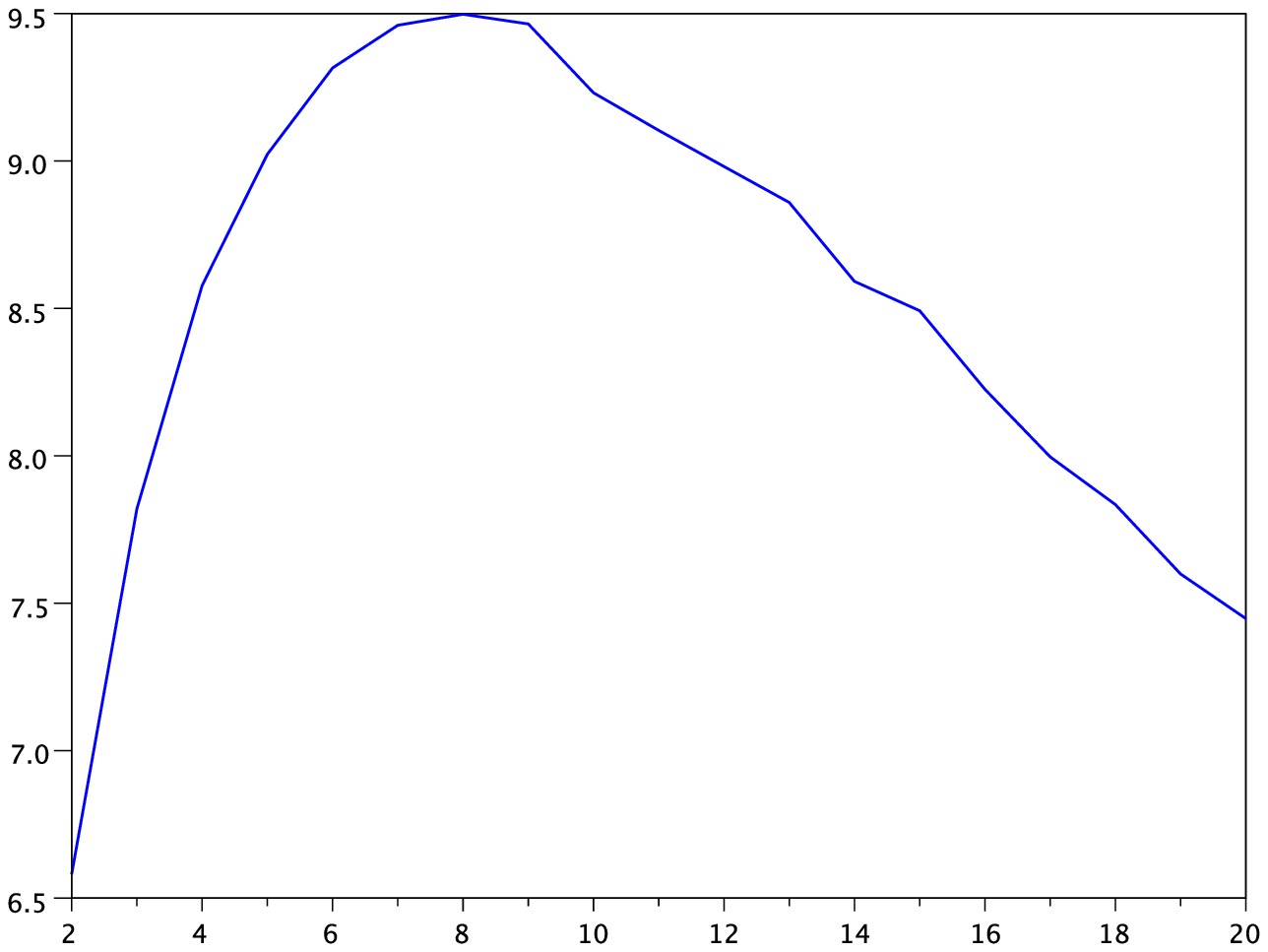}\hspace*{0.4cm}}
\caption {Averaged number of steps versus the dimension $d$, $\varepsilon=0.001$, $10\,000$ trials, $t=1$.}
 \label{fig4.1}
\end{figure}
\clearpage
\appendix
\section{Technical results}
We first start with Jensen's inequality:
\begin{lemma}\label{lem:appen:Jensen} Let $X$ and $Y$ be two random variables and $\mathcal{A}$ a $\sigma$-algebra, then
\[
\mathbb{E}[\max(X,Y)]\ge \max (\mathbb{E}[X],\mathbb{E}[Y])\quad \mbox{and}\quad \mathbb{E}[\max(X,Y)|\mathcal{A}]\ge \max (\mathbb{E}[X|\mathcal{A}],\mathbb{E}[Y|\mathcal{A}]).
\]
\end{lemma}
\begin{proof} We shall just prove the first inequality. The proof of the second one is similar. We get
\[
\mathbb{E}[\max(X,Y)]=\mathbb{E}[\mathbb{E}[\max(X,Y)|X]]\ge \mathbb{E}[\max(X,\mathbb{E}[Y|X])].
\]
Since $\max(X,\mathbb{E}[Y|X])\ge X$, we deduce that $\mathbb{E}[\max(X,\mathbb{E}[Y|X])]\ge \mathbb{E}[X]$. On the other side $\max(X,\mathbb{E}[Y|X])\ge \mathbb{E}[Y|X]$ and therefore $\mathbb{E}[\max(X,\mathbb{E}[Y|X])]\ge \mathbb{E}[\mathbb{E}[Y|X]]=\mathbb{E}[Y]$. Combining both inequalities leads to the result.
\end{proof}
Let us now present properties concerning a particular probability distribution arising in the random walk on moving spheres.
\begin{lemma}
\label{lem:born}
Let $W:=-2\log\Big(1-\sqrt{\frac{2e}{d}}\psi_d(R)\Big)-\log(1-R)$ where the function $\psi_d$ is defined by \eqref{eq:psi} and $R$ is a random variable with the following probability density function:
\[
f_R(s)=\frac{1}{\Gamma(d/2)}\frac{\psi_d^d(s)}{s}1_{[0,1]}(s)=\frac{1}{s\Gamma(d/2)}\Big(s\log (s^{-d/2})\Big)^{d/2}1_{[0,1]}(s). 
\]
Then $W$ has its two first moments (denoted by $\mathcal{M}_1$ and $\mathcal{M}_2$) bounded.
\end{lemma}
\begin{proof}
Let us first note that $W$ is a non-negative random variable, since $R$ and $\sqrt{\frac{2e}{d}}\psi_d(R)$ are $[0,1]$-valued. If we denote $W_a=-2\log\Big(1-\sqrt{\frac{2e}{d}}\psi_d(R)\Big)$ and $W_b:=-\log(1-R)$, then it suffices to prove that $\mathbb{E}[W_a^2]<\infty$ and $\mathbb{E}[W_b^2]<\infty$.
\begin{align*}
\mathbb{E}[W_b^2]=\int_0^1f_b(s)\,ds\quad\mbox{with}\quad f_b(s)=\frac{(\log(1-s))^2}{s\Gamma(d/2)} \Big(s\log (s^{-d/2})\Big)^{d/2}.
\end{align*}
Let us observe that $f_b(s)$ tends to $0$ as $s\to 0$ and in a neighborhood of $1$, $f_b(s)\sim C_1 (1-s)^{d/2}(\log(1-s))^2$ where $C_1>0$ is a constant. We deduce that $f_b$ is integrable on the whole interval $[0,1]$ which implies that $\mathbb{E}[W_b^2]<\infty$. For $W_a$ we get
\[
\mathbb{E}[W_a^2]=\int_0^1f_a(s)\,ds\quad\mbox{with}\quad f_a(s)=\frac{4\log^2\Big(1-\sqrt{\frac{2e}{d}\, s\log(s^{-d/2})}\Big)}{s\Gamma(d/2)} \Big(s\log (s^{-d/2})\Big)^{d/2}.
\]
In a neigborhood of $0$, we have $f_a(s)\sim C_2 s^{d/2}(\log s)^{d/2+1}$, in a neighborhood of $1$, we observe $f_a(s)\sim C_3 (1-s)^{d/2+1}$ and finally in a neigborhood of $1/e$, $f_a(s)\sim C_4 \log^2|s-\frac{1}{e}|$. We deduce that $f_a$ is integrable on the whole interval $[0,1]$ and $\mathbb{E}[W_a^2]<\infty$.
\end{proof}

\section{Improvements for the classical random walk on spheres}
\label{sec:appen}
In this section, we focus our attention to the classical random walk on spheres. We consider an $0$-thick domain $\mathcal{D}$, see the definition developed in \eqref{eq:thick}, and the Euclidean distance to the boundary $d(x)=\delta(x,\partial \mathcal{D})$. The random walk is then defined as follows: 
\begin{itemize}
\item we start with an initial condition $X_0$ and fix two parameters $\varepsilon>0$ and $\beta\in]0,1[$. 
\item While $d(X_n)>\varepsilon$, we construct
\begin{equation}\label{eq:def:class}
X_{n+1}=X_n+\beta d(X_n)\gamma_n
\end{equation}
where $(\gamma_n)$ stands for a sequence of independent random variables uniformly distributed on the unit sphere in $\mathbb{R}^d$. 
\end{itemize}
We adapt here several results of \cite{Binder-Bravermann} to our particular situation. Let us recall that $U$ is the \emph{energy function} defined by \eqref{eq:def:U} which is based on the set of measures $\mathcal{M}$, defined by \eqref{eq:cond:mu}, and on the Riesz potential. Since $\mathcal{D}$ is a $0$-thick domain, the following Lemma holds. 
\begin{lemma}\label{lem:B}
There exist two constants $\delta>0$ and $\eta>0$, such that: for any $y\in\mathcal{D}$ (we define $x$ the closest point of $y$ belonging to the boundary) and any measure $\mu\in\mathcal{M}$, we have:
\begin{enumerate}
\item either $U(z)>U^\mu(z)+1$ whenever $\Vert z-x\Vert <\delta d(y)$ and $d(z)>\delta/4\, d(y)$
\item or $\mu(B(y,2d(y)))\ge \eta d(y)^d$.
\end{enumerate}
\end{lemma}
This lemma, which is quite general and is not directly linked to the random walk, has an important consequence on it (for the proof of Lemma \ref{lem:B}, see \cite{Binder-Bravermann}). 
\begin{proposition}\label{prop:classic} Let us consider $\delta$ the constant of Lemma \ref{lem:B} and the random walk $(X_n)_{n\ge 0}$  defined by \eqref{eq:def:class} with $\beta\in]1-\delta/2,1-\delta/4[$. Then there exists a constant $L>0$ such that the sequence  $U_n:=U(X_n)$ satisfies
\[
\mathbb{E}[U_{n+1}-U_n|U_n]>L,\quad\forall n\ge 0.
\]
\end{proposition}
In \cite{Binder-Bravermann}, the authors consider a general random walk defined by \eqref{eq:def:class}. They prove that there exist an interger $k$ and a constant $L>0$ such that $\mathbb{E}[U_{n+k}-U_n|U_n]>L$. Here we adapt the proof by introducing a particular condition on the parameter $\beta$ which permits in fact to set $k=1$. 
\begin{proof}
Let us consider $X_n$. Due to the weak compactness of the set of measures $\mathcal{M}$ (see Remark \ref{rem}), there exists a measure $\mu\in\mathcal{M}$ such that
\[
U_n=U(X_n)=U^\mu(X_n).
\]
For this particular measure, either the first or the second point of the previous lemma are satisfied. \\
{\bf Step 1.} Let us assume that the first point is satisfied that is, 
$U(z)>U^\mu(z)+1$ when $\Vert z-x\Vert<\delta d(y)$ and $d(z)>\delta/4\, d(y)$. 
Since $U^\mu(X_n)$ is a submartingale, we get
\begin{align*}
\mathbb{E}[U_{n+1}-U_n|X_n]&=\mathbb{E}[U_{n+1}-U^\mu_{n+1}|X_n]+\mathbb{E}[U^\mu_{n+1}-U^\mu_n|X_n]\\
&\ge \mathbb{E}[U_{n+1}-U^\mu_{n+1}|X_n]\\
&\ge \mathbb{P}\Big(\Vert X_{n+1}-x_n\Vert<\delta d(X_n), \, d(X_{n+1})>(\delta/4) d(X_n) \Big| X_n\Big),
\end{align*}
where $x_n$ is the closest point of $X_n$ on the boundary $\partial \mathcal{D}$. We denote by $u_n=\frac{x_n-X_n}{d(X_n)}$ which belongs to the unit sphere.
Using the definition of the random walk and the particular choice of the parameter $\beta$, we get immediately
\[
d(X_{n+1})> (1-\beta)d(X_n)>\delta/4 d(X_n),
\]
and
\begin{align}\label{eq:majj}
\Vert X_{n+1}-x_n\Vert&= d(X_n)\Vert u_n-\beta \gamma_n\Vert\le  d(X_n)\Big((1-\beta)\Vert \gamma_n\Vert+\Vert\gamma_n-u_n\Vert\Big)\nonumber \\
&= d(X_n)(1-\beta+\Vert\gamma_n-u_n\Vert)<  d(X_n)(\delta/2+\Vert\gamma_n-u_n\Vert).
\end{align}
Let us recall that $u_n$ is a unit vector. Then we define the set $\Gamma_{u_n}$ of points $u$ belonging to  the unit sphere of dimension $d$  such that $\Vert u-u_n\Vert<\delta/2$. Let us just note that $\Gamma_{u_n}$ is a non empty open set. We observe that $\mathbb{P}(\gamma_n\in\Gamma_{u_n})=:p>0$ for any $n\ge 0$ and does not depend on $u_n$ due to rotational invariant of the distribution. Furthermore, for any $\gamma_n\in\Gamma_{u_n}$, \eqref{eq:majj} implies that $\Vert X_{n+1}-x_n\Vert < \delta d(X_n)$. Therefore
\begin{align*}
\mathbb{E}[U_{n+1}-U_n|X_n]&\ge \mathbb{P}\Big(\Vert X_{n+1}-x_n\Vert<\delta d(X_n), \, d(X_{n+1})>(\delta/4) d(X_n) \Big| X_n\Big)\\
&\ge \mathbb{P}(\gamma_n\in\Gamma_{u_n})=p>0.
\end{align*} 
{\bf Step 2.}  The second case concerns the condition 
\[
\mu(B(y,2d(y)))\ge \eta d(y)^d.
\]
By the Green formula, for a $\mathcal{C}^2$-smooth function $h$,
\begin{align*}
\mathbb{E}[h(X_{n+1})|X_n]-h(X_n)&=\int_{\mathbb{S}(X_n,\beta d(X_n))}h(y)d\sigma(y)-h(X_n)\\
&=\int_0^{\beta d(X_n)} r^{1-d}\int_{B(X_t,r)}\Delta h (y) dV(y)\,dr.
\end{align*}
Since $\Delta U^\mu(y)=2(d+2)\int_0^\infty \frac{\mu(B(y,r))}{r^{d+3}}\,dr$ outside the support of the measure $\mu$ (consequently $U^\mu$ is a $\mathcal{C}^2$-function in the domain $\mathcal{D}$), then, for any $y$ satisfying $\Vert y-X_n\Vert \le \beta d(X_n)$, we get
\begin{align*}
\Delta U^\mu(y)&\ge 2(d+2)\mu(B(X_n,2d(X_n)))\int_{(2+\beta)d(X_n)}^\infty\frac{dr}{r^{d+3}}=\frac{2\mu(B(X_n,2d(X_n))}{((2+\beta)d(X_n))^{d+2}}\\
&\ge \frac{2\eta}{(2+\beta)^{d+2}}\, d(X_n)^{-2}.
\end{align*}
Applying the previous results to the particular regular function $h=U^\mu$, we deduce:
\[
\mathbb{E}[U_{n+1}-U_n|X_n]=\mathbb{E}[U^\mu(X_{n+1})|X_n]-U^\mu(X_n)\ge Cd(X_n)^{-2}\int_0^{\beta d(X_n)}r\,dr=\frac{C\beta^2}{2},
\]
for some positive constant $C$ depending on $\eta$, $\beta$ and $d$. Since $\mathbb{E}[U_{n+1}-U_n|U_n]=\mathbb{E}[U_{n+1}-U_n|X_n]$, we obtain the announced result.
\end{proof}

\noindent

\bibliographystyle{plain}

\begin{thebibliography}{}

\end{thebibliography}


\begin{thebibliography}{10}

\bibitem{Bass}
Richard~F. Bass.
\newblock {\em Diffusions and elliptic operators}.
\newblock Probability and its Applications (New York). Springer-Verlag, New
  York, 1998.

\bibitem{Binder-Bravermann}
Ilia Binder and Mark Braverman.
\newblock The rate of convergence of the walk on spheres algorithm.
\newblock {\em Geom. Funct. Anal.}, 22(3):558--587, 2012.

\bibitem{deaconu-herrmann-2016}
M.~Deaconu and S.~Herrmann.
\newblock {S}imulation of hitting times for {B}essel processes with non integer
  dimension.
\newblock {\em Bernoulli}, in print.

\bibitem{deaconu-herrmann-2015}
M.~Deaconu, Maire S., and S.~Herrmann.
\newblock {T}he walk on moving spheres: a new tool for simulating {B}rownian
  motion's exit time from a domain.
\newblock {\em Math. Comp. Sim.}, in print.

\bibitem{deaconu-herrmann-2013}
Madalina Deaconu and Samuel Herrmann.
\newblock Hitting time for {B}essel processes---walk on moving spheres
  algorithm ({W}o{MS}).
\newblock {\em Ann. Appl. Probab.}, 23(6):2259--2289, 2013.

\bibitem{evans-2010}
Lawrence~C. Evans.
\newblock {\em Partial differential equations}, volume~19 of {\em Graduate
  Studies in Mathematics}.
\newblock American Mathematical Society, Providence, RI, second edition, 2010.

\bibitem{Friedman}
Avner Friedman.
\newblock {\em Partial differential equations of parabolic type}.
\newblock Prentice-Hall, Inc., Englewood Cliffs, N.J., 1964.

\bibitem{Friedman2}
Avner Friedman.
\newblock {\em Stochastic differential equations and applications}.
\newblock Dover Publications, Inc., Mineola, NY, 2006.
\newblock Two volumes bound as one, Reprint of the 1975 and 1976 original
  published in two volumes.

\bibitem{Gurov2001}
T.~Gurov, P.~Whitlock, and I.~Dimov.
\newblock {\em A Grid Free Monte Carlo Algorithm for Solving Elliptic Boundary
  Value Problems}, pages 359--367.
\newblock Springer Berlin Heidelberg, Berlin, Heidelberg, 2001.

\bibitem{K-S}
Ioannis Karatzas and Steven~E. Shreve.
\newblock {\em Brownian motion and stochastic calculus}, volume 113 of {\em
  Graduate Texts in Mathematics}.
\newblock Springer-Verlag, New York, second edition, 1991.

\bibitem{mascagni-hwang-2003}
Michael Mascagni and Chi-Ok Hwang.
\newblock {$\epsilon$}-shell error analysis for ``walk on spheres'' algorithms.
\newblock {\em Math. Comput. Simulation}, 63(2):93--104, 2003.

\bibitem{muller_56}
Mervin~E. Muller.
\newblock Some continuous {M}onte {C}arlo methods for the {D}irichlet problem.
\newblock {\em Ann. Math. Statist.}, 27:569--589, 1956.

\bibitem{Oksendal}
Bernt {\O}ksendal.
\newblock {\em Stochastic differential equations}.
\newblock Universitext. Springer-Verlag, Berlin, sixth edition, 2003.
\newblock An introduction with applications.

\bibitem{sabelfeld-simonov-1994}
K.~K. Sabelfeld and N.~A. Simonov.
\newblock {\em Random walks on boundary for solving {PDE}s}.
\newblock VSP, Utrecht, 1994.

\bibitem{Sabelfeld}
Karl~Karlovich Sabelʹfelʹd.
\newblock {\em Monte Carlo methods in boundary value problems}.
\newblock Springer Verlag, 1991.

\bibitem{villa-2012}
Jos{\'e} Villa-Morales.
\newblock On the {D}irichlet problem.
\newblock {\em Expo. Math.}, 30(4):406--411, 2012.

\bibitem{villa-2016}
Jos{\'e} Villa-Morales.
\newblock Solution of the {D}irichlet problem for a linear second-order
  equation by the monte carlo method.
\newblock {\em Comm. Stoch. Anal.}, 10(1):83--95, 2016.

\end{thebibliography}

\end{document}